\documentclass[11pt,a4paper]{amsart}
\usepackage[T1]{fontenc}
\usepackage{lmodern}
\usepackage{amsmath,amssymb,amsthm,mathtools}
\usepackage[margin=30mm]{geometry}
\usepackage{microtype}
\usepackage[hidelinks]{hyperref}
\hypersetup{pdftitle={Analyticity and tautness of smooth Dupin submanifolds},
 pdfsubject={Research article},
 pdfkeywords={Dupin submanifold, real analyticity, tautness, Morse-Bott, finiteness, Nash isotopy}}
\newcommand{\R}{\mathbb R}
\newcommand{\Z}{\mathbb Z}
\newcommand{\G}{\mathcal G}
\newcommand{\Pcal}{\mathcal P}
\newcommand{\Hess}{\operatorname{Hess}}
\newcommand{\grad}{\operatorname{grad}}
\newcommand{\Crit}{\operatorname{Crit}}
\newcommand{\tr}{\operatorname{tr}}
\newcommand{\rank}{\operatorname{rank}}
\newcommand{\Res}{\operatorname{Res}}
\newcommand{\im}{\operatorname{im}}

\newcommand{\dist}{\operatorname{dist}}
\newtheorem{maintheorem}{Theorem}

\newtheorem{theorem}{Theorem}[section]
\newtheorem{proposition}[theorem]{Proposition}
\newtheorem{lemma}[theorem]{Lemma}
\newtheorem{corollary}[theorem]{Corollary}
\theoremstyle{remark}
\newtheorem{remark}[theorem]{Remark}
\numberwithin{equation}{section}
\title[Analyticity and tautness of smooth Dupin submanifolds]
{Analyticity and tautness of smooth Dupin submanifolds}
\date{}

\author{Stephan Wiesendorf}
\begin{document}
\begin{abstract}
We prove that a smooth immersed Euclidean submanifold is locally
Nash if it satisfies the Dupin condition on the regular part of
its unit normal bundle. This condition is equivalent to the
Morse--Bott property of all squared distance functions and, for
compact connected embedded submanifolds, to tautness over the
field $\mathbb{Z}_2$. No local finiteness assumption on the
regular locus, or regularity assumption on its complement, is
required. We also prove that, in each fixed Euclidean space,
every compact connected smooth taut submanifold can be joined
to one of finitely many representatives by a Nash isotopy through
taut embedded submanifolds. 
In each fixed intrinsic dimension, we also obtain only
finitely many diffeomorphism types admitting taut embeddings
in Euclidean spaces.
We first treat hypersurfaces by combining
a dimension-dependent degree bound on proper Dupin pieces with
a continuation argument using finite jets and resultants. An
analytic Hessian criterion then gives the Morse--Bott property,
and small tubes extend the conclusions to arbitrary codimension.
The degree bound also provides uniform semialgebraic complexity.
Together with the semialgebraicity of tautness in parameter
families and Nash triviality, this yields the finiteness and
deformation results. Classical classification theorems give
further topological restrictions.
\end{abstract}
\maketitle

\section{Introduction and main results}

Focal geometry provides a local description of tautness. Let
$M$ be a submanifold of a complete Riemannian manifold $N$, and
write $\eta=\exp^\perp:\nu M\to N$ for its normal
exponential map. The focal multiplicity at a normal vector $v$
is $\dim\ker d\eta_v$. On the regular focal locus in the
sense of Hebda, the focal vectors locally form a smooth
hypersurface on which this multiplicity is constant. If the
multiplicity is at least two, the kernel is automatically tangent
to this hypersurface and integrates to leaves contained in the
fibers of $\eta$ \cite{Hebda}. See also
\cite[Theorem~4.16 and Remark~4.17]{BP} for a proof in the more
general Finsler setting. Thus tangency on the regular focal locus
is an additional condition only in multiplicity one.

Once tangency holds on the regular focal locus, the remaining
issue is to extend this integrability across singular focal
vectors. For a closed embedded submanifold of a
complete Riemannian manifold, the author proved that integrability
of the full kernels of the normal exponential map, with local
fibers smooth and tangent to these kernels, is equivalent to
the Morse--Bott property of all endpoint energy functionals
and to $\Z_2$-tautness \cite[Theorem~2.9]{Wiesendorf}.
In Euclidean space, we show that the Dupin condition on the
regular part of the unit normal bundle suffices for this
extension. It also forces local Nash regularity, without any
compactness or embeddedness assumption. For compact connected
embedded submanifolds, it therefore characterizes tautness.

The Dupin condition connects principal curvature geometry with
this focal description. A curvature surface with constant nonzero
principal curvature lies in a critical set of a squared distance
function and lifts to an integral manifold of the kernel of the
normal exponential map. Pinkall proved that taut Euclidean submanifolds
are Dupin \cite{Pinkall}. Miyaoka obtained this implication
independently for hypersurfaces \cite{Miyaoka}. Compact embedded
proper Dupin hypersurfaces are taut, as Thorbergsson showed in \cite{Thorbergsson}. Here \emph{proper} means that the
number of distinct principal curvatures is constant. This does not
include all taut hypersurfaces, and the converse implication for
general Dupin hypersurfaces requires control at points where principal
curvatures collide. Cecil's survey \cite[Sections 3--4]{Cecil} discusses this question and the distinction between curvature surfaces on the regular locus and those tangent to full principal spaces at singular points. The survey also provides historical background and an extensive bibliography on Dupin geometry and tautness.

Cecil--Chi--Jensen proved local analytic algebraicity for smooth
proper Dupin hypersurfaces and established analyticity for general
Dupin hypersurfaces under an additional local finiteness hypothesis
\cite{CCJ}.
Chi's algebraicity theorem for compact taut submanifolds \cite{Chi}
starts from tautness and therefore does not supply this missing
regularity implication. Our hypersurface results remove the local
finiteness hypothesis, establish tautness in the compact embedded
case, and give finiteness up to deformation through Dupin
hypersurfaces. Small tubes extend the regularity and
Morse--Bott conclusions to arbitrary codimension. The same
construction, together with semialgebraic parameter families,
also gives finiteness through taut embeddings in each fixed
Euclidean space.

Let $f:M^m\to\R^n$, $1\le m<n$, be a smooth immersion of a
manifold without boundary, and let $\nu^1M$ be its unit normal
bundle. At $(p,\xi)\in \nu^1M$, write $A_\xi$ for the shape
operator and $X^h$ for the horizontal lift of $X\in T_pM$ with
respect to the normal connection. Let $\G_\nu\subset \nu^1M$
be the open dense set where the number of distinct eigenvalues
of $A_\xi$ is locally constant. On this set we require
\begin{equation*}\tag{$D_\nu$}\label{eq:Dnu}
 d\lambda(X^h)=0\ \
 \text{ whenever }\ A_\xi X=\lambda X.
\end{equation*}
Here $\lambda$ is a local eigenvalue function on $\G_\nu$.
For hypersurfaces this is the usual regular Dupin condition.
In arbitrary codimension it is satisfied by Dupin submanifolds
in the unit normal bundle sense of Pinkall
\cite{Pinkall,CCJpre}, as recalled in
\cite[Remark~3.3]{Cecil}. We impose no condition at points outside
$\G_\nu$, and the normal bundle need not be flat. The Codazzi
equation applied to small tubes shows that \eqref{eq:Dnu} is
automatic for eigenvalues of multiplicity at least two on
$\G_\nu$, as follows from the tube formulas in
Section~\ref{sec:tubes}. Thus the additional condition again
concerns only the simple eigenvalues on the regular locus.

The relation with focal geometry can be seen directly from the
normal endpoint map. The principal curvatures in the unit normal
direction $\xi$ are the eigenvalues of $A_\xi$. In Euclidean space,
the normal exponential map is
\[
 \eta:\nu M\longrightarrow\R^n,\qquad
 \eta(p,v)=f(p)+v.
\]
In the horizontal--vertical splitting
$T_{(p,v)}\nu M\simeq T_pM\oplus\nu_pM$ induced by the normal
connection, the Weingarten formula gives
\[
 d\eta_{(p,v)}(X,V)=df_p((I-A_v)X)+V.
\]
Since $A_{t\xi}=tA_\xi$, it follows that
\[
 \ker d\eta_{(p,t\xi)}
 =\ker(I-tA_\xi)\oplus\{0\}.
\]
Thus $t\xi$ is a focal vector precisely when
$t=1/\lambda$ for a nonzero principal curvature $\lambda$
of $A_\xi$. Its image $f(p)+t\xi$ is the corresponding focal
point, and the focal multiplicity equals the multiplicity
of $\lambda$. For a nonzero local eigenvalue function
$\lambda$ on $\G_\nu$, condition \eqref{eq:Dnu} says precisely
that $\ker d\eta$ is tangent to the local focal hypersurface
parametrized by $(p,\xi)\mapsto(p,\xi/\lambda(p,\xi))$.
\par\vspace{\baselineskip}

Our first main result shows that condition \eqref{eq:Dnu},
imposed only on the regular locus, forces local Nash regularity
and the Morse--Bott property of all squared distance functions.
For compact connected embedded submanifolds, it also yields
$\Z_2$-tautness.

\begin{maintheorem}\label{thm:general}
Let $f:M^m\to\R^n$, $1\le m<n$, be a $C^\infty$ immersion
satisfying \eqref{eq:Dnu}.
\begin{enumerate}
\item Every local image branch is real-analytic and locally
Nash. There is a compatible analytic structure on $M$ for which
$f$ is analytic, and every squared distance function
$L_q=\tfrac12|f-q|^2$ is Morse--Bott.
\item If $M$ is compact and connected and $f$ is an embedding,
then $f$ is $\Z_2$-taut.
\end{enumerate}
\end{maintheorem}

In fact, the Morse--Bott conclusion of
Theorem~\ref{thm:general} also admits a converse,
yielding the following characterization of the
regular Dupin condition.

\begin{corollary}\label{cor:characterization}
Let $f:M^m\to\R^n$, $1\le m<n$, be a smooth immersion of a
manifold without boundary. Then \eqref{eq:Dnu} holds if and only
if every squared distance function $L_q=\tfrac12|f-q|^2$ is
Morse--Bott. If $M$ is compact and connected and $f$ is an
embedding, these conditions are equivalent to $\Z_2$-tautness.
\end{corollary}

The proof of Theorem~\ref{thm:general} is given in
Section~\ref{sec:tautness}, where we also prove
Corollary~\ref{cor:characterization}. The proof first treats
hypersurfaces, using a degree bound and a continuation argument,
and then passes to arbitrary codimension by small tubes.

Let $f:M^m\to\R^{m+1}$, $m\ge1$, be a $C^\infty$ immersion.
All manifolds are without boundary. Write $g$ for the induced
metric and $A$ for the shape operator with respect to a local
unit normal $\nu$, with the convention $D_X\nu=-df(AX)$.
Let $\G\subset M$ be the set where the number of distinct eigenvalues of $A$ is locally constant. The set $\G$ is open and dense. Indeed, that number is
lower semicontinuous and bounded by $m$. A point maximizing it in any
nonempty open set has a neighborhood on which it is constant.
On $\G$, separated spectral projections, the distinct principal
curvatures, and their multiplicities are smooth.

We assume the following \emph{regular Dupin condition}:
\begin{equation}\tag{D}\label{eq:D}
 d\lambda(X)=0\qquad
 \text{for }X\in E_\lambda:=\ker(A-\lambda I)\ \text{ on }\G.
\end{equation}
This condition is independent of the choice of local unit normal. We impose no condition on curvature surfaces at points outside \(\mathcal G\). On \(\mathcal G\), the principal distributions are integrable by the Codazzi equation. Codazzi also implies that a principal curvature of multiplicity at least two is constant along its curvature leaves. Thus, Condition (D) imposes an additional restriction only on the simple principal curvatures. Under this hypothesis, every connected component of \(\mathcal G\) is a proper Dupin hypersurface.

The hypersurface regularity result is proved in
Theorem~\ref{thm:A}. In addition, it gives a bound depending
only on $m$ for the degree of an irreducible polynomial vanishing
on the image of each connected component of $M$.

Analyticity here concerns the image branches and the compatible
analytic structure, rather than an arbitrary smooth
parametrization. Recall that a Nash function is a smooth semialgebraic function on an open semialgebraic
set and that such functions are known to be real-analytic \cite[Chapter~8]{BCR}.

Theorem~\ref{thm:A} supplies the regularity needed to study
distance functions at eigenvalue collisions. Using analyticity and the Dupin condition, we show that every squared distance function is Morse--Bott. This property carries over to the corresponding endpoint energy functionals, allowing us to apply the author's characterization of tautness to deduce Proposition~\ref{thm:B}.

A smooth immersion $f: N \to \R^n$ of a compact manifold is
$\mathbb F$-\emph{taut} if every Morse squared distance function
\[
 L_q(x)=\tfrac12|f(x)-q|^2
\]
is perfect over the field $\mathbb F$, i.e., for every $k$
it has exactly $b_k(N;\mathbb{F})$ critical points of index $k$,
where $b_k(N;\mathbb{F})$ denotes the $k$-th
$\mathbb{F}$-Betti number. 

Ozawa proved that distance
functions on compact taut submanifolds are Morse--Bott
\cite{Ozawa}, and Terng--Thorbergsson established the corresponding
statement for endpoint energy functionals \cite[Theorem~2.8]{TT}.
The converse is the characterization recalled above. Moreover,
tautness over any field implies $\Z_2$-tautness
\cite[Theorem~2.11]{Wiesendorf}. In Euclidean space, the Morse--Bott condition can equivalently be stated in terms of squared distance functions and the
endpoint energy splits as the squared distance function plus a
positive quadratic form. We use this splitting in the proof of
Proposition~\ref{thm:B}. For a detailed discussion of tautness, we refer the reader to \cite{Wiesendorf}, and for a historical context and an extensive list of references, to \cite{Cecil} and \cite{TT}.

Our results answer two questions raised in the literature. Cecil--Chi--Jensen asked whether every connected Dupin hypersurface has the local finiteness property \cite[Question 36]{CCJpre}. Theorem~\ref{thm:A} implies that it does. Indeed, the shape operator has Nash coefficients in local Nash coordinates. In each such chart, the loci of fixed principal multiplicities and the regular set are semialgebraic. The regular set therefore has only finitely many connected components in sufficiently small semialgebraic neighborhoods.
Further, Cecil asked whether compact embedded Dupin submanifolds are necessarily taut, without assuming that the number of distinct principal curvatures is constant \cite[p.~136]{Cecil1997}. Theorem~\ref{thm:general} answers this question in arbitrary codimension. We require the Dupin condition only on the regular set. 
For hypersurfaces, the semi-Dupin condition additionally
requires curvature surfaces tangent to every principal
space, including at singular points.
Proposition~\ref{prop:distances} establishes the existence
of these surfaces under \eqref{eq:D}.
Together with Proposition~\ref{thm:B}, this proves the
remaining implication in the conjectured equivalence of
tautness and the semi-Dupin condition for compact connected
embedded hypersurfaces recalled in \cite[Remark~4.3]{Cecil}.
In Section~\ref{sec:tautness}, we explain their relation to the fibers of the normal exponential map, prove Proposition~\ref{thm:B}, and then pass to arbitrary codimension.

The second main result is a finiteness theorem that retains
tautness throughout the deformation. In each fixed Euclidean
space, every compact connected smooth taut submanifold can be
joined to one of finitely many representatives by a Nash isotopy
through taut embedded submanifolds
(Theorem~\ref{fin:thm:tautdeformations}). Kuiper's bound on
substantial ambient dimension then gives finiteness of abstract
diffeomorphism types in each fixed intrinsic dimension
(Corollary~\ref{fin:cor:intrinsic}). These statements apply in
arbitrary codimension.

The degree bound provides the link between the local regularity
argument and these global finiteness results. A compact embedded
hypersurface contained in an irreducible algebraic hypersurface
is the closure of a union of regular algebraic components.
Hardt's theorem \cite{Hardt} then gives uniform semialgebraic
complexity, and Coste--Shiota's Nash triviality and finiteness
theorems \cite{CS} apply. We show that tautness defines a
semialgebraic subset of the resulting parameter spaces. By
restricting to this subset, we retain tautness, and hence the
Dupin condition, throughout the deformation. The general
statement is as follows, with tautness understood over $\Z_2$.

\begin{maintheorem}
\label{fin:thm:tautdeformations}
For each $n\ge2$, there is a finite list of compact connected
taut Nash submanifolds of $\R^n$ such that every compact
connected smooth taut submanifold of $\R^n$ is Nash isotopic
to a member of the list through taut embedded submanifolds.
The isotopy admits a smooth ambient extension, allowing a
smooth reparametrization of time.

Consequently only finitely many smooth ambient isotopy types and
diffeomorphism types occur, and the total $\Z_2$-Betti number is
bounded by a constant depending only on $n$.
\end{maintheorem}

By Theorem~\ref{thm:general}, this applies in particular to
compact connected embedded submanifolds satisfying \eqref{eq:Dnu}.
The isotopies concern the embedded images with their induced
Nash structures. They are jointly Nash families of embeddings
$\iota_t:M\to\R^n$ with $\iota_0$ the inclusion. All intermediate
images are taut, hence Dupin. Principal multiplicities need not
remain fixed, and Theorem~\ref{fin:thm:tautdeformations} does not
assert finiteness up to Lie-sphere transformations.

Independently of these finiteness results,
Theorem~\ref{thm:general} and classical classification theorems for
taut embeddings yield rigidity for homology spheres in arbitrary
codimension and determine the possible diffeomorphism types in
dimensions two and three. In dimension four, vanishing of the
first or second $\Z_2$-Betti number gives further classifications.
For hypersurfaces, the homology of a product of two spheres
characterizes Dupin cyclides. This last conclusion does not extend
to arbitrary codimension.

Since \eqref{eq:D} is automatically met for principal curvatures of
multiplicity at least two, we obtain the following corollaries.
In the compact case the immersion is shown to be embedded
before Proposition~\ref{thm:B} is applied.

\begin{corollary}\label{cor:analytic}
Let $f:M^m\to\R^{m+1}$, $m\ge2$, be a $C^\infty$ immersion.
If every principal curvature has multiplicity at least two on
$\G$, all conclusions of Theorem~\ref{thm:A} hold.
\end{corollary}

\begin{corollary}\label{cor:taut}
Let $f:M^m\to\R^{m+1}$, $m\ge2$, be a $C^\infty$ immersion
of a compact connected manifold. If every principal curvature
has multiplicity at least two on $\G$, then $f$ is a
$\Z_2$-taut embedding.
\end{corollary}

Sections~\ref{sec:degree} and~\ref{sec:continuation} establish
the degree bound and continuation argument for
Theorem~\ref{thm:A}. Section~\ref{sec:tautness} proves
Proposition~\ref{thm:B} and the reduction of Theorem~\ref{thm:general}
to the hypersurface case, and Section~\ref{sec:multiplicity}
treats the multiplicity corollaries.
Section~\ref{fin:sec:carriers} proves the finiteness results,
including Theorem~\ref{fin:thm:tautdeformations} and its
hypersurface specialization, Corollary~\ref{thm:C}.
The final two sections give
the topological consequences and discuss examples and scope.

\medskip
\noindent\textbf{Acknowledgements.}
I am grateful to Alexander Lytchak for his valuable feedback on an early version of this manuscript and for his continued interest, support, and insightful mathematical discussions over many years.

This article was produced in collaboration with ChatGPT, an AI
assistant developed by OpenAI. The collaboration included
developing proof strategies, refining mathematical arguments,
cross-referencing known examples and results with the relevant
literature and drafting the manuscript. The author assumes full
responsibility for the entire content of the article.

\section{A uniform degree bound on the regular pieces}\label{sec:degree}

We first recall the curvature-surface parametrization and finite-jet construction for Dupin hypersurfaces of Cecil--Chi--Jensen \cite[Section~3]{CCJ},
which they trace to an earlier proof sketch of Pinkall.
We formulate this construction in semialgebraic families
and extract an algebraic degree bound depending only on
the dimension. It turns out that this degree bound is the key step to control the description of Dupin hypersurfaces in the neighborhood of singular strata through zero sets of irreducible polynomials, even with infinitely many regular components.

\subsection{Rational curvature surfaces}
It is well known that the relevant geometric information on a
proper Dupin piece of a hypersurface can be obtained from the
shape operator $A$, its eigenvalues and their derivatives via
algebraic operations (see \cite{CCJ}). In particular, these data
provide locally rational parametrizations of curvature surfaces. For the sake of completeness, we briefly recall this description. Let $\lambda_1<\cdots<\lambda_s$ be the
distinct principal curvatures, with multiplicities $m_i=\dim E_i$, and put
\[
 \Pi_i=\prod_{j\ne i}\frac{A-\lambda_jI}{\lambda_i-\lambda_j},
\]
for the projection onto the $i$-th curvature distribution. For $X,Y\in E_i=\im\Pi_i$ and $Z\in E_j$, the Codazzi equation gives
\begin{equation}\label{eq:codazzi-leaf}
 (\lambda_i-\lambda_j)g(\nabla_XY,Z)
       =Z(\lambda_i)g(X,Y).
\end{equation}
In particular, $E_i$ is integrable and a direct computation using \eqref{eq:codazzi-leaf} shows that the Euclidean second
fundamental form of each leaf is $g(X,Y)\eta_i$, where
\begin{equation}\label{eq:eta}
 \eta_i=\lambda_i\nu+
       \sum_{j\ne i}\frac{df(\Pi_j\grad\lambda_i)}{\lambda_i-\lambda_j}
\end{equation}
is the curvature normal vector of the leaf $L_i$ of the curvature distribution $E_i$. The curvature leaves are open subsets of round spheres or affine planes
\cite[Section 3]{CCJ}. Note that this assertion includes one-dimensional
curvature surfaces, because of the Dupin condition. For multiplicity at least two this follows automatically from Codazzi, even without the Dupin condition. The Dupin condition is merely a requirement for the multiplicity one case on the regular part.
The sphere or plane through $f(p)$ with tangent space $df(E_i(p))$
and curvature vector $\eta_i(p)$ has the standard rational parametrization
\begin{equation}\label{eq:leaf-map}
 \Phi_i(p,v)=f(p)+
 \frac{v+\frac12|v|^2\eta_i(p)}{1+\frac14|\eta_i(p)|^2|v|^2}
 \ \ \text{ for }  v\in df(E_i(p)).
\end{equation}
For $\eta_i\ne0$, its center is $f(p)+\eta_i/|\eta_i|^2$ and
its radius is $1/|\eta_i|$. For $\eta_i=0$, the formula describes
the affine plane. The differential with respect to $v$ at zero is
the identity. Thus \eqref{eq:leaf-map} parametrizes a neighborhood
of $f(p)$ in its curvature leaf.

\subsection{Semialgebraic propagation of jets}
We first review the geometric jet construction and its
propagation due to Cecil--Chi--Jensen \cite{CCJ}, before
discussing some semialgebraic consequences that lead to
a universal degree bound.

For graph coordinates $F(x)=(x,h(x))$ we obtain the following expressions depending on $h$: 
\begin{equation}\label{eq:graph-data}
 g_{ab}=\delta_{ab}+h_ah_b,\qquad
 w=(1+|Dh|^2)^{1/2},\qquad
 \nu=\frac{(-Dh,1)}w,\qquad A=\frac1w g^{-1}D^2h.
\end{equation}
Fix an ordered multiplicity pattern $(m_1,\ldots,m_s)$, with
$\sum m_i=m$. The distinct eigenvalues are uniquely determined
semialgebraically by
\[
 \det(tI-A)=\prod_{i=1}^s(t-\lambda_i)^{m_i}
 \ \ \text{ with } \ \ \lambda_1<\cdots<\lambda_s.
\]
Their derivatives are determined by the spectral identity
\begin{equation}\label{eq:eigen-derivative}
 \partial_\ell\lambda_i=\frac1{m_i}\tr(\Pi_i\partial_\ell A).
\end{equation}
Thus successive differentiation, together with the polynomial expression
for $\Pi_i$, expresses every finite jet of $\lambda_i$ and $\Pi_i$
by finitely many algebraic operations on a finite jet of $h$ and the
eigenvalue values. The only spectral denominators are nonzero
differences $\lambda_i-\lambda_j$. This works in the graph frame,
which need not be orthonormal.

Choose columns $\Pi_ie_{j_1},\ldots,\Pi_ie_{j_{m_i}}$  of the matrix representing $\Pi_i$ forming a
basis of $E_i$ at $x_0$, and choose a coordinate inclusion
$B_i\colon\R^{m-m_i}\to\R^m$ whose image complements $E_i(x_0)$.
Both choices are specified by nonvanishing minors, with finitely many
possibilities depending only on $m$. The fact that all considered constructions can be described by finitely many formulas, independently of the concrete hypersurface, is crucial to obtain our global degree bound below.

For $a \in \R^{m_i}$ and $y \in \R^{m-m_i}$ near the origin, define
\begin{align*}
 v(a,y)&=\sum_{\ell=1}^{m_i}a_\ell\,
 dF_{x_0+B_iy}(\Pi_i e_{j_\ell}),\\
 Z_i(a,y)&=\Phi_i(x_0+B_iy,v(a,y)).
\end{align*}
Here the first argument of $\Phi_i$ denotes the point on the graph.
Let $\pi$ be projection onto the first $m$ coordinates. The
differential of $\pi\circ Z_i$ at $(0,0)$ is invertible, because its columns
are the chosen basis of $E_i$ and the chosen complement. Since
$Z_i$ lies on the hypersurface, it yields a local parametrization. These parametrizations supplement the curvature surface parameters given by \eqref{eq:leaf-map} with transverse coordinates providing a transport of jets locally along the curvature surfaces as described below. 

If we set $Z_i^{m+1}:= \operatorname{pr}_{m+1}\circ Z_i$, where $\operatorname{pr}_{m+1}:\R^{m+1} \to \R$ denotes the projection onto the last coordinate, we have
\[Z_i(a,y)=((\pi\circ Z_i)(a,y),Z_i^{m+1}(a,y)).\]

The $k$-jet of $Z_i$ at $(a,0)$ is obtained by algebraic operations
from $a$ and $j^{k+3}_{x_0}h$. Indeed, an $r$-jet of $\eta_i$
requires at most an $(r+3)$-jet of $h$, by
\eqref{eq:eta}--\eqref{eq:eigen-derivative}. If $\pi Z_i$ is
locally invertible at $(a,0)$, the identity
\[
 h=Z_i^{m+1}\circ(\pi\circ Z_i)^{-1}
\]
along the graph $F(x)=(x,h(x))$ therefore recovers the $k$-jet of $h$ at the new base point. Jets of an inverse
map are rational expressions in its jets, with powers of its
Jacobian determinant as denominators, because derivatives of an inverse map are rational expressions of the derivatives of the original map. For $k=0$, only the value of
$Z_i$ is required. Consequently, for each choice of minors there is
a semialgebraic transformation
\begin{equation}\label{eq:jet-map}
 T_{i,k}\colon(a,j^{k+3}_{x_0}h)\longmapsto j^k_{x_a}h
 \ \ \text{ with } x_a=(\pi\circ Z_i)(a,0),
\end{equation}
on the domain where the specified eigenvalue pattern, minors, and
denominators are valid. We also include the base point in the jet data.

The important observation is that the jet at $x_a$ can be obtained from the jet at $x_0$ by finitely many algebraic operations and that these transformations are determined by universal finite-dimensional formulas, i.e., formulas that are independent of the concrete hypersurface of fixed dimension.
To define them on formal jet data, fix an order of differentiation in
\eqref{eq:eigen-derivative} and impose the indicated algebraic
conditions. 
On jets of actual proper Dupin graphs, these formulas agree
with the geometric construction. We use the same formulas
on semialgebraic domains of formal jet data. In particular,
no semialgebraicity assertion about the actual set of jets
realizable by solutions of the Dupin equations is needed.

To describe the general space of admissible jets more precisely, let $\Sigma_m$ denote the set of finite indices $\sigma$ recording the multiplicity pattern and
choices of minors in the discussion above, and note that we have $|\Sigma_m|<\infty$. 

For every $\sigma \in \Sigma_m$ we have a semialgebraic map  
\[\Psi_\sigma: D_\sigma \to \R^{m+1},\]
where $D_\sigma \subset J^{3s}(\R^m,\R)\times\R^{m_1}\times\ldots\times\R^{m_s}$ is the corresponding semialgebraic domain. Starting with a $3s$-jet, the map $\Psi_\sigma$ is defined by a composition of
\[
 T_{1,3s-3},\ T_{2,3s-6},\ \ldots,\ T_{s,0},
\]
where the involved transformations are determined by $\sigma$. 

Fixing the initial jet $J$, we obtain the set
\[S_{\sigma,J}=\{\Psi_\sigma(J,a) \ | \ (J,a)\in D_\sigma\},\]
which is a semialgebraic subset of $\R^{m+1}$, because it is the semialgebraic image of a semialgebraic set by the theorem of Tarski--Seidenberg. These sets are the fibers over the initial jets of the set 
\[S_{\sigma}=\{(J,z) \ | \ \exists a: (J,a)\in D_\sigma \wedge z=\Psi_\sigma(J,a)\},\]
that is also semialgebraic by the theorem of Tarski--Seidenberg. 

For an actual proper Dupin graph, zero steps preserve the corresponding
jets. Therefore the derivative in the $a$-variables at zero, with $\sum_i\dim a_i=m$, has
image
$df(E_1)\oplus\cdots\oplus df(E_s)$ and rank $m$.
Therefore, for a choice of $\sigma$ valid at the base point,
the restriction of $a\mapsto\Psi_\sigma(J,a)$ to a sufficiently
small neighborhood of $a=0$ parametrizes a neighborhood of
the chosen point in the hypersurface. This neighborhood is
contained in the full fiber image $S_{\sigma,J}$, which is
all that is needed for the degree bound below.

Our next goal is to show that all components of proper Dupin pieces of a hypersurface are annihilated by an irreducible polynomial with globally bounded degree. To this end, we first observe that all fibers $S_{\sigma,J}$ are contained in the zero set of a polynomial with bounded degree. For the following discussion we use the common identification of the space of jets $J^{3s}(\R^m,\R)$ with $\R^\ell$. 

\begin{lemma}\label{lem:family-degree}
Let $\mathcal D\subset\R^\ell\times\R^m$ be semialgebraic and
let $\Psi\colon\mathcal D\to\R^{m+1}$ be semialgebraic. There
is an integer $D_S$ such that each nonempty fiber image
$S_J=\{\Psi(J,a)\ | \ (J,a)\in\mathcal D\}$ is contained in the
zero set of a nonzero polynomial of degree at most $D_S$.
\end{lemma}
\begin{proof}
Quantifier elimination gives a finite combination of sign conditions on polynomials $R_\tau(J,z)$, so that we obtain a Boolean description of the set 
$S= \{(J,z)\ | \ z\in S_J\}$ in terms of finitely many polynomials
\cite[Chapter~2]{BCR}, i.e., the Boolean rule $\mathcal B(R_1(J,z),\ldots,R_N(J,z))$ determines whether a point $(J,z)$ belongs to the set.

After fixing $J$ with $S_J\ne\emptyset$, the set $S_J$ has
dimension at most $m$, because it is the image of a semialgebraic
subset of $\R^m$. As a subset of $\R^{m+1}$, it therefore has
empty interior. Some of the specialized polynomials
$R_j(J,\cdot)$ may vanish identically. We replace their sign
conditions in the Boolean formula by the corresponding constant
truth values and consider only the remaining polynomials.
If none of these vanished at a point of $S_J$, all their signs
would remain unchanged on a small ball around that point.
The Boolean formula would then hold throughout the ball,
contradicting the empty interior of $S_J$. Thus, at every point
of $S_J$, at least one of the remaining polynomials vanishes.
Their product is a nonzero polynomial vanishing on $S_J$.
Its degree is at most $\sum_j\deg_zR_j$, independently of $J$.
\end{proof}

\begin{proposition}\label{prop:degree}
For each $m$ there is an integer $D(m)$ such that every connected
proper Dupin graph in $\R^{m+1}$ is annihilated by a nonconstant
irreducible real polynomial of degree at most $D(m)$.
\end{proposition}
\begin{proof}
By Lemma~\ref{lem:family-degree} we obtain an assignment $\Sigma_m \to \mathbb N, \sigma \mapsto D_{S_\sigma}$. Since there are only finitely many
families $S_\sigma$, there is a largest resulting
degree bound $D(m)$. 
A sufficiently small neighborhood of any point of a proper
Dupin graph is contained in one of these fiber images, so
it satisfies a nonzero polynomial relation with this bound.
The graph is analytic by
\cite[Theorem~4]{CCJ}. On a connected graph neighborhood,
one irreducible factor of that polynomial vanishes
identically and the analytic identity theorem extends this
factor's relation to the entire connected graph. Its degree
is no larger than that of the original polynomial.
\end{proof}

\section{Continuation across the singular locus}\label{sec:continuation}

In this section we prove the analyticity part of Theorem \ref{thm:general} in the case of a Dupin hypersurface, to which we will reduce the general case later by a tube construction.

\begin{theorem}\label{thm:A}
Let $f:M^m\to\R^{m+1}$ be a $C^\infty$ immersion satisfying
\eqref{eq:D}. Every local image branch is real-analytic and locally
Nash. There is an analytic structure compatible with the smooth
structure on $M$ for which $f$ is analytic.

Moreover, there is an integer $D(m)$ depending only on $m$ such
that the image of each connected component of $M$ is contained
in the zero set of a nonconstant irreducible real polynomial of
degree at most $D(m)$.
\end{theorem}

To prove Theorem \ref{thm:A}, we have to continue the analytic structure over the singular locus. This will be done by controlling the space of polynomials of suitable degree which vanish on proper Dupin pieces. The next two lemmas concern arbitrary smooth graphs and isolate
the continuation step from the Dupin geometry. Let $\Pcal_D$ be
the real vector space of polynomials in $(x_1,\ldots,x_m,y)=(x,y)$ of
total degree at most $D$, and put
\[
 N_D=\dim\Pcal_D=\binom{m+1+D}{m+1}.
\]
We fix a constant $K$ that is large enough, say $K=2D^2$, and consider a smooth real-valued function $h$ on an open set $U\subset\R^m$. Define a linear map
\begin{equation}\label{eq:T}
 \mathcal T_x\colon\Pcal_D\longrightarrow
       \R^{\{\beta\  | \ |\beta|\le K\}}\ \ \text{ by } \ \ 
 Q\longmapsto\bigl(\partial^\beta[Q(\cdot,h)](x)\bigr)_{|\beta|\le K}.
\end{equation}
In fixed bases this is a smooth matrix-valued function of $x$. The algebraic key observation for our continuation argument in Lemma \ref{lem:gluing} is that if an irreducible polynomial $P \in \Pcal_D$ vanishes on a proper Dupin piece, then every element of $\Pcal_D$ that also vanishes on that piece has to be a multiple of $P$. The following lemma describes this observation pointwise.

\begin{lemma}[Finite-order contact]\label{lem:contact}
Suppose a nonconstant irreducible polynomial $P \in \Pcal_D$ vanishes up to order $K$ at $x_0$, i.e., $P \in \ker\mathcal T_{x_0}$.
Then
\begin{equation}\label{eq:kernel}
 \ker\mathcal T_{x_0}=(P)\cap\Pcal_D
 \ \ \text{ and } \ \ \rank\mathcal T_{x_0}=N_D-N_{D-d} \ \text{ for } d = \deg P.
\end{equation}
\end{lemma}
\begin{proof}
The polynomial $P$ depends on $y$, since a nonzero polynomial
in $x$ alone cannot be annihilated by $\mathcal T_{x_0}$, because $K > D$.
Since $P$ is irreducible in $\R[x][y]$ and depends on $y$, it is primitive with respect to $y$ and therefore also irreducible over $\R(x)$ by Gauss's lemma. Now assume 
$Q\in\ker\mathcal T_{x_0}$ and $P\nmid Q$. Then the resultant, i.e., the determinant of the Sylvester matrix of $P$ and $Q$,
\[
 R(x)=\Res_y(P,Q)=U(x,y)P(x,y)+V(x,y)Q(x,y)
\]
is a nonzero polynomial, with $U,V$ polynomial. The Sylvester
matrix has size at most $2D$ and entries of $x$-degree at
most $D$, so $\deg R\le2D^2=K$. This argumentation also works if $Q$ is independent
of $y$ with $R=Q^{\deg_yP}$.

Substitution of $y=h(x)$ kills the
$P$ term as well as the $Q$ term up to order at least $K$.
By the general Leibniz rule, $R$ vanishes to order greater than its degree, which 
contradicts our assumption. Hence $P\mid Q$. The converse inclusion in
\eqref{eq:kernel} follows immediately from the general Leibniz rule. Finally, multiplication by $P$
identifies $\Pcal_{D-d}$ with $(P)\cap\Pcal_D$, giving
the rank formula.
\end{proof}

Together with Proposition \ref{prop:degree}, we can extend the conclusions of Lemma \ref{lem:contact} from regular points to their components.

\begin{corollary}\label{cor:extend}
Let $x\mapsto(x,h(x))$, $x\in U$, be a hypersurface graph
satisfying \eqref{eq:D}, and let $C_\alpha\subset U$ be a
connected component of its regular set. Let $P_\alpha$ be
the irreducible polynomial of degree $d_\alpha\le D$
associated with the graph over $C_\alpha$ by
Proposition~\ref{prop:degree}. 
Then, for every \(x\in C_\alpha\), we have $\ker \mathcal T_x=(P_\alpha)\cap\mathcal P_D$. 
In particular, the map
$C_\alpha\longrightarrow \operatorname{Gr}_{N_{D-d_\alpha}}(\Pcal_D), \, x\longmapsto\ker \mathcal T_x$, is constant, and for every \(x\in C_\alpha\) we have
\[
\operatorname{rank}\mathcal T_x
=
\dim\mathcal P_D-\dim\mathcal P_{D-d_\alpha}.
\]

\end{corollary}
\begin{proof}
By Proposition~\ref{prop:degree}, there exists an irreducible
polynomial $P_\alpha$ that vanishes identically on the graph
over $C_\alpha$. Consequently, \(P_\alpha\in\ker \mathcal T_x\) for every \(x\in C_\alpha\). Applying Lemma \ref{lem:contact} at each such point yields
\[
\ker \mathcal T_x=(P_\alpha)\cap\mathcal P_D.
\]
Since \(P_\alpha\) is independent of \(x\), the kernel is constant throughout \(C_\alpha\). The rank formula follows from Lemma \ref{lem:contact}.
\end{proof}

Now we are able to prove the following continuation Lemma as our technical key result to obtain the main result of this section, Theorem \ref{thm:A}.

\begin{lemma}[Continuation of bounded-degree relations]\label{lem:gluing}
Let $U\subset\R^m$ be connected and open, and let
$h\in C^\infty(U)$. Suppose $G\subset U$ is open and dense
and that, on each connected component $C_\alpha$ of $G$, the
graph of $h$ is annihilated by a nonconstant irreducible polynomial
$P_\alpha$ with $\deg P_\alpha\le D$. Then one such polynomial
vanishes on the graph over all of $U$, and $h$ is locally Nash.
\end{lemma}
\begin{proof}
By assumption, for every connected component $C_\alpha$ there is an irreducible polynomial $P_\alpha \in \Pcal_D$ with $P_\alpha(x,h(x)) = 0$ and $\deg P_\alpha = d_\alpha$ for all $x \in C_\alpha$. By Lemma \ref{lem:contact} we have $\ker \mathcal T_x = (P_\alpha)\cap \Pcal_D$ and $\rank \mathcal T_x = N_D - N_{D-d_\alpha}$.

Since the degrees $d_\alpha$ are integers between $1$ and $D$,
they have a maximum, even if $G$ has infinitely many connected
components. Choose $P_*$ among the $P_\alpha$ with maximal
degree $d_*$, and put $r_*=N_D-N_{D-d_*}$.
Lemma~\ref{lem:contact} gives
\[
 \rank\mathcal T_x=N_D-N_{D-\deg P_\alpha}\le r_*
\]
on all components $C_\alpha$. It follows that all $(r_*+1)$-minors vanish on the dense set $G$, hence
on $U$, because the maps $\mathcal T_x$ depend continuously, even smoothly, on $x \in U$. Thus, we have $\rank\mathcal T\le r_*$ everywhere on $U$.

Let $\Omega$ be the connected component of
$\operatorname{int}_U\{x\in U \ | \ P_*(x,h(x))=0\}$ containing the
regular component on which $P_*$ was chosen. We want to show $U = \Omega$ to finish the proof. To prove this claim, it is sufficient that there are no relative boundary points of $\Omega$ in $U$, since $U$ is connected and $\Omega$ is open.
For $x_0\in\partial_U\Omega$, every derivative of
$P_*(x,h(x))$ vanishes at $x_0$, by continuity of the derivatives. Therefore $P_\ast \in\ker \mathcal T_{x_0}$ and we obtain from Lemma~\ref{lem:contact} $\ker \mathcal T_{x_0} = (P_\ast)\cap \Pcal_D$ and
$\rank\mathcal T_{x_0}=r_*$. Because the rank is a lower semicontinuous function with maximal value $r_\ast$, the rank of $\mathcal T_x$ is constant on a small connected
neighborhood $V$ of $x_0$. 

To deduce that $\Omega = U$, we further need the constancy of the kernels on $V$.
We therefore fix an inner product on $\Pcal_D$, and denote by $\mathcal K(x): \Pcal_D \to \ker\mathcal T_x$ 
the orthogonal projection onto $\ker\mathcal T_x$, which depends smoothly on $x\in V$. On each regular
component, \eqref{eq:kernel} identifies the kernel with
the fixed space $(P_\alpha)\cap\Pcal_D$. Consequently
$d\mathcal K=0$ on $G\cap V$, hence on $V$ by density of $G\cap V$ in $V$. 
Thus $\mathcal K$ is constant on $V$, i.e., $\ker \mathcal T_x = \ker \mathcal T_{x_0}$ for all $x \in V$. Since
$P_*\in\ker\mathcal T_{x_0} = \ker\mathcal T_{x}$, the zeroth component of
\eqref{eq:T} gives $P_*(x,h(x))=0$ on all of $V$. Thus, the graph over $V$ is contained in the zero set of $P_\ast$. As $x_0$ is a boundary point of $\Omega$, the set $V$
meets $\Omega$, so that the union $\Omega \cup V$ is an open, connected set completely contained in the set $\{x \in U\ | \ P_\ast(x,h(x))=0\}$. By definition, this implies $V\subset\Omega$, which contradicts the choice of $x_0$ as a boundary point. Therefore, having no relative boundary, $\Omega$ is also closed in $U$, which implies $\Omega=U$ by connectedness of $U$.

To obtain Nash regularity at singular points of the
polynomial relation, we use the classical fact that a
$C^\infty$-smooth semialgebraic function on an open
semialgebraic set is real analytic, i.e., Nash
\cite[Chapter~8]{BCR}.

Write $a(x)$ for
the leading coefficient and $\Delta(x)$ for the discriminant of $P_*$ as a polynomial in
$y$. Since $P_*$ depends on $y$ and is
irreducible, it is also irreducible over the characteristic-zero field $\R(x)$, thus separable over \(\R(x)\). Consequently, its discriminant with respect to \(y\) does not vanish identically, i.e., \(\Delta(x)=\operatorname{Disc}_y(P_*)\not\equiv0.\) Likewise, we have \(a(x)\not\equiv0.\)

We define the exceptional algebraic set
\[
E=\{x\ | \ a(x)\Delta(x)=0\}.
\]
Since \(a\Delta\) is a nonzero polynomial, \(E\) has empty interior. Consequently, \(U\setminus E
\) is open and dense in \(U\), and for \(x_0\in U\setminus E\) we have \(a(x_0)\neq0\) and \(\Delta(x_0)\neq0.
\) Thus, the polynomial \(P_*(x_0,\cdot)\) has only simple complex roots in this case. In particular, for every real root \(y_0\), we have
\[
\frac{\partial P_*}{\partial y}(x_0,y_0)\neq0,
\]
and we can therefore apply the analytic implicit function theorem. This implies that in a neighborhood of \(x_0\), each real root extends to a real-analytic function \(y=\varphi(x)\) satisfying
\[
P_*(x,\varphi(x))=0.
\]
Since these branches are locally separated and their graphs are defined by an algebraic equation together with suitable inequalities distinguishing the individual roots, they are also semialgebraic.

Thus, locally, we obtain finitely many Nash functions \(\varphi_1,\ldots,\varphi_q,
\) whose values are precisely the real roots of \(P_*(x,\cdot)\). Since \(h\) is continuous and satisfies \(P_*(x,h(x))=0,
\) it must locally coincide with one of these branches. In particular, \(h\) is locally Nash on \(U\setminus E\). 

It remains to deal with points in \(E\), where roots may coincide or the leading coefficient may vanish. Let us therefore  consider a point \(x_0\in E\). Choose a sufficiently small open ball \(B\) around \(x_0\) whose closure is contained in \(U\). Since \(B\) is semialgebraic and \(E\) is defined by a polynomial equation, \(B\setminus E\) is also semialgebraic. So, this set has only finitely many connected components \(C_1,\ldots,C_N,
\) each of which is semialgebraic. 

On each component \(C_j\), all roots of \(P_*(x,\cdot)\) are simple. Hence its real roots can be ordered as \(\varphi_{j,1}(x)<\cdots<\varphi_{j,q_j}(x)\). Each of these functions is real-analytic and semialgebraic. On a connected component \(C_j\), the function \(h\) cannot switch from one root branch to another, since distinct branches do not intersect there. Consequently, for each \(j\), there exists an index \(s_j\) such that \( h|_{C_j}=\varphi_{j,s_j}.\)

The graph of \(h\) over \(B\setminus E\) is therefore
\[
\Gamma_0
=\bigcup_{j=1}^N
\{(x,\varphi_{j,s_j}(x))\ |\ x \in C_j\},
\]
which is a finite union of semialgebraic sets and is therefore itself semialgebraic. Since \(E\) has empty interior, \(B\setminus E\) is dense in \(B\). Over \(B\), the graph is therefore given as
\[\overline{\Gamma_0}\cap(B\times\mathbb R).
\]
We now use the fact that the closure of a semialgebraic
set is semialgebraic. Since \(\Gamma_0\) is semialgebraic, so is \(\overline{\Gamma_0}\cap(B\times\mathbb R)\), proving that \(h|_B\) is smooth and semialgebraic, and therefore Nash as mentioned above.

\end{proof}

Now we are able to prove our main result of this section.

\begin{proof}[Proof of Theorem~\ref{thm:A}]
Represent a local image branch as a smooth graph over a
connected open set. Each connected regular piece is
proper Dupin, and Proposition~\ref{prop:degree}
provides irreducible polynomial relations of uniformly bounded
degree. Lemma~\ref{lem:gluing} therefore makes the entire graph
locally Nash. No restriction on the number of regular
components is needed.

These graph charts give an analytic atlas compatible
with the smooth structure, because the transition maps of these charts are
compositions of analytic graph parametrizations with
affine ambient projections. Since the charts are obtained from the smooth local immersion theorem, the immersion is analytic
in this atlas. On a connected component $M_0$ of $M$ 
choose an arbitrary point \(p\in M_0\) and a local Nash graph chart around \(p\). By the preceding argument, there exists an irreducible polynomial \(P\) such that
\[
P(x,h(x))=0.
\]
Initially, this polynomial is expressed in the Euclidean coordinates of the chosen graph chart. Transforming back to the original ambient coordinates by an invertible affine transformation, we obtain a polynomial
\[
\widetilde P\in\mathbb R[\widetilde x_1,\ldots,\widetilde x_{m+1}].
\]
Both its degree and irreducibility are preserved under this transformation, so
\[
\deg\widetilde P=\deg P\le D(m).
\]
Thus, on some nonempty open neighborhood \(W\subset M_0\), we have
\(
\widetilde P(f(q))=0\) for \( q\in W\).

Now consider the function
\(
\psi=\widetilde P\circ f:M_0\longrightarrow\mathbb R
\).
Since \(f\) is real-analytic with respect to the analytic structure constructed above and \(\widetilde P\) is a polynomial, \(\psi\) is real-analytic. Moreover, \(\psi\) vanishes on the nonempty open subset \(W\).

By the identity theorem for real-analytic functions, a real-analytic function on a connected real-analytic manifold that vanishes on a nonempty open subset must vanish identically. It follows that \(\psi\equiv0\) on \(M_0\).
Consequently, \(\widetilde P(f(q))=0\) for all \(q\in M_0\).
\end{proof}

\section{Morse--Bott regularity and tautness}\label{sec:tautness}

To prove that the distance functions are Morse--Bott, we will use the analytic criterion in Lemma~\ref{lem:MB} below. For a function $h$ on
a Riemannian manifold $(N,g)$, let $H$ be the self-adjoint
endomorphism defined by $g(HX,Y)=\Hess h(X,Y)$. An
eigenvector is called simple when its eigenvalue has
multiplicity one.

\begin{lemma}\label{lem:MB}
Let $(N^n,g)$ be a real-analytic Riemannian manifold, let
$h\colon N\to\R$ be real-analytic, and let $x_0\in N$ be
a critical point of $h$. Suppose there is a neighborhood
of $x_0$ and a constant $C$ such that every unit simple eigenvector
$e$ of $H$ at every point of this neighborhood satisfies
\begin{equation}\label{eq:MB-criterion}
 \bigl|(\nabla_e\Hess h)(e,e)\bigr|\le C|dh|_g.
\end{equation}
Then $h$ is Morse--Bott near $x_0$.
\end{lemma}
\begin{proof}
If $H(x_0)$ is nonsingular, $x_0$ is a nondegenerate critical
point and the conclusion follows. Suppose, then, that
$r:=\dim\ker H(x_0)>0$. The analytic version of the Morse splitting lemma gives coordinates
$(u,v)\in\R^r\times\R^{n-r}$ centered at $x_0$ in which
\begin{equation}\label{eq:splitting}
 h=h(x_0)+\rho(u)+\tfrac12v^{\mathsf T}Bv
 \ \ \text{ with } \ \ j^2_0\rho=0 \ \text{ and } \ g(0)=I,
\end{equation}
where $\rho$ is analytic and $B$ is symmetric and invertible. If \(r=n\), the \(v\)-block and the corresponding coordinate changes below are omitted. 

To prove the claim, we will show that $\rho$ vanishes identically. But before, we first explain why the splitting can be chosen as above, with the
normalization $g(0)=I$. To this end, start with coordinates whose coordinate
vectors at $x_0$ form an orthonormal basis adapted to
$\ker H(x_0)$ and its orthogonal complement. 
Since the analytic function \(F(u,v):=h_v(u,v)\) satisfies $F(0,0) = 0$ and \(D_vF(0,0)=h_{vv}(0,0)=B\), the implicit
function theorem solves $F(u,v)= h_v(u,v)=0$ as $v=\chi(u)$, with
$\chi(0)=D\chi(0)=0$ and $\chi$ analytic. Writing $w:=v-\chi(u)$, we have for \(\widetilde h(u,w):=h(u,w+\chi(u))\) with \(\widetilde h_w(u,0)=0\) by Taylor's
formula with respect to $w$ 
\[
\widetilde h(u,w)
=\widetilde h(u,0)
+\frac12w^{\mathsf T}\widetilde B(u,w)w,
\]
and therefore
\[
\widetilde h(u,w)
=h(u,\chi(u))
+\frac12w^{\mathsf T}\widetilde B(u,w)w.
\]
From this, we obtain for \(\rho(u):=h(u,\chi(u))-h(0,0)\)  
\[
 \widetilde h(u,w)=h(x_0)+\rho(u)+\tfrac12w^{\mathsf T}\widetilde B(u,w)w \
 \  \text{ with } \ \widetilde B(0,0)=B.
\]
To make the matrix in the quadratic term constant, we set $S:=B^{-1}(\widetilde B-B)$. Then $S$ is analytic with $BS= \widetilde B-B$, resp. $\widetilde B = B(I+S)$, and $S(0,0) = B^{-1}(\widetilde B(0,0)-B)= 0$, since $\widetilde B(0,0) = B$. Thus, we can define
$R:=(I+S)^{1/2}$ by its convergent matrix power series near
the origin, i.e., by 
\[
\begin{aligned}
R
&=I+\frac12S-\frac18S^2+\frac1{16}S^3+\ldots\\
&=\sum_{k=0}^{\infty}\binom{1/2}{k}S^k.
\end{aligned}
\] 
Because of $S^{\mathsf T} = (\widetilde B-B)B^{-1}$ by symmetry of $B$ and $\widetilde B$, we obtain the identity $S^{\mathsf T}B=BS$, which means that $S$ is self-adjoint with respect to the bilinear form defined by $B$.
Further, repeated application of the last identity implies $(S^k)^{\mathsf T}B=BS^k$ and therefore
$R^{\mathsf T}BR=\widetilde B$. Thus the analytic change
$w\mapsto R(u,w)w$ makes the quadratic term constant.
Both coordinate changes have identity differential at zero,
so they preserve $g(0)=I$ and provide the expression claimed above:  
\[
h(u,v)=h(0,0)+\rho(u)+\frac12v^{\mathsf T}Bv,
\]
reusing \(h\) for the transformed function and \(v\) for the final coordinates.

This construction also works when $B$ is indefinite. When the
$v$ block is absent, no change is needed, as we already mentioned above.

Now, suppose that $\rho$ does not vanish identically near zero. We will construct simple eigenvectors of the Hessian near the origin along which the third covariant derivative of \(h\) violates the bound \eqref{eq:MB-criterion}.

By analyticity and $j_0^2\rho = 0$, \(\rho\) has a first nonzero homogeneous term $P_d$,
of degree $d\ge3$, i.e., \(\rho(u)=P_d(u)+O(|u|^{d+1})\) with 
\(P_d(tu)=t^dP_d(u)\) and $d\ge 3$. In particular, we have 
\[
D^k\rho(u)=D^kP_d(u)+O(|u|^{d-k+1}) \ \ \text{ for }  k\le 3.\]

Replacing $h$ by $-h$ if necessary does
not change the hypothesis, and allows us to choose a point
$u_0$ on the unit sphere in $\R^r$ at which $P_d$ attains a positive
maximum $a=P_d(u_0)>0$. Since $u_0$ maximizes $P_d$ on the unit sphere, its gradient
is orthogonal to the tangent space, i.e., 
\(
DP_d(u_0)=c\,u_0
\) for some real number $c$. Euler's identity for a homogeneous polynomial of degree $d$
states that
\(
\langle DP_d(u),u\rangle=dP_d(u),
\)
hence
\begin{equation}
DP_d(u_0)=da\,u_0.
\label{eq:gradient-at-maximum}
\end{equation}
Differentiating Euler's identity with respect to $u$ gives
\( D^2P_d(u)\,u=(d-1)DP_d(u).
\) At $u=u_0$, equation
\eqref{eq:gradient-at-maximum} therefore yields
\[
D^2P_d(u_0)u_0=d(d-1)a\,u_0.
\]

Set
\(
\mu:=d(d-1)a\) and  \(C_0:=D^2P_d(u_0).
\)
Then $C_0u_0=\mu u_0$. To show that $\mu$ is simple,
let $w\perp u_0$. Since $u_0$ maximizes $P_d$ on the
unit sphere, the second derivative test along a great
circle through $u_0$ with initial velocity $w$ gives
\[
w^{\mathsf T}D^2P_d(u_0)w\le da|w|^2.
\]
Because $C_0$ is symmetric and $u_0$ is an eigenvector,
$u_0^\perp$ is $C_0$-invariant. The preceding inequality
shows that all eigenvalues of $C_0$ on $u_0^\perp$ are
at most $da$. Since
\(
\mu=d(d-1)a>da
\) for $d \ge 3$,
the eigenvalue $\mu$ is simple and separated from
the remaining spectrum.

We now consider the points $z_s=(su_0,0)$ as
$s\searrow0$. At these points, $dh=O(s^{d-1})$ and in the chosen coordinates \eqref{eq:splitting}, the metric matrix satisfies
$g_s:=g(z_s)=I+O(s)$. 
By the Taylor expansion above and the homogeneity
of $P_d$, we have
\[
D^2\rho(su_0)=s^{d-2}C_0+O(s^{d-1}).
\]
Since the mixed second derivatives vanish and the
$v$-block of $D^2h$ is the constant matrix $B$,
the matrix of the Hessian form at $z_s$ is
\begin{equation}\label{eq:hessian-block}
K_s=
\begin{pmatrix}
s^{d-2}C_0+O(s^{d-1}) & O(s^{d-1})\\
O(s^{d-1}) & B+O(s^{d-1})
\end{pmatrix}.
\end{equation}
Here the connection terms are $O(s^{d-1})$,
since they are linear in $dh$ and the Christoffel
symbols are bounded. Because of $H_s=g_s^{-1}K_s$, the eigenvalue equation
for $H_s$ can be written as
\(
H_se=\theta e
 \Longleftrightarrow
K_se=\theta g_se
\).
Motivated by the leading term \(s^{d-2}C_0\), we put \(\theta=s^{d-2}\tau\) and seek a solution with \(\tau\to\mu\) as \(s\to0\). 
For bounded $\tau$ and sufficiently small $s$, the
lower-right block of $K_s-s^{d-2}\tau g_s$ is
invertible, since it converges to $B$.
Moreover, its off-diagonal blocks are $O(s^{d-1})$,
because those of $K_s$ are $O(s^{d-1})$ and
those of $g_s$ are $O(s)$.

This implies that an eigenvector corresponding to an
eigenvalue of the form $\theta=s^{d-2}\tau$, with
bounded $\tau$, has only a small component in the
transverse $v$-directions. To see this, set $Q:=K_s-s^{d-2}\tau g_s$ and write
the eigenvalue equation in block form as
\[
\begin{pmatrix}
Q_{uu}&Q_{uv}\\
Q_{vu}&Q_{vv}
\end{pmatrix}
\begin{pmatrix}
e_u\\e_v
\end{pmatrix}=0.
\]
Since $Q_{vv}$ is invertible, the second block
equation gives
\(
e_v=-Q_{vv}^{-1}Q_{vu}e_u.
\)
As $Q_{vv}^{-1}=O(1)$ and $Q_{vu}=O(s^{d-1})$,
we obtain
\[
|e_v|=O(s^{d-1})|e_u|.
\]
Thus the transverse component of the eigenvector
vanishes to order at least $d-1$ as $s\searrow0$,
provided $e_u$ remains bounded. Substituting this expression for $e_v$ into the
upper block equation yields
\[
\left(Q_{uu}-Q_{uv}Q_{vv}^{-1}Q_{vu}\right)e_u=0.
\]
By the block estimates above, we have
\[
Q_{uu}=s^{d-2}(C_0-\tau I+O(s))
\ \ 
\text{ and } 
\ \
Q_{uv}Q_{vv}^{-1}Q_{vu}=O(s^{2d-2}).
\]
Dividing the reduced equation by $s^{d-2}$ therefore
gives
\[
(C_0-\tau I+O(s))e_u=0.
\]
After this rescaling, the reduced equation extends
analytically to $s=0$. Indeed, for $s>0$, define
\[
M(s,\tau)=
s^{-(d-2)}
\left(Q_{uu}-Q_{uv}Q_{vv}^{-1}Q_{vu}\right).
\]
The apparent singularity at $s=0$ is removable,
since analyticity and the preceding estimates imply
that $Q_{uu}$ is divisible by $s^{d-2}$,
while $Q_{uv}Q_{vv}^{-1}Q_{vu}$ vanishes to order
at least $2d-2$. Hence $M$ extends analytically
to $s=0$ and satisfies
\[
M(0,\tau)=C_0-\tau I.
\]
The eigenvalue $\mu$ is a simple eigenvalue of $C_0$. Thus
we can apply the analytic implicit function theorem to 
$F(s,\tau)=\det M(s,\tau)$ at $(0,\mu)$, because \(F(0,\tau)=\det(C_0-\tau I)\) with $F(0,\mu) = 0$ and $\partial_\tau F(0,\mu)\neq0$ in this case.
This yields an
analytic function $\tau(s)=\mu+O(s)$ satisfying
$F(s,\tau(s))=0$.

The kernel of $M(s,\tau(s))$ is one-dimensional
and admits an analytic choice of vectors satisfying
$e_u(s)=u_0+O(s)$. The previously established relation
\[
e_v=-Q_{vv}^{-1}Q_{vu}e_u
\]
then gives $e_v=O(s^{d-1})$. After normalization
with respect to $g_s$, we obtain a simple eigenvalue
of $H_s$ and a corresponding unit eigenvector satisfying
\begin{equation}\label{eq:eigen-asymptotics}
\theta_s=\mu s^{d-2}+O(s^{d-1}) \ \ \text{ and } \ \
e_s=(u_0+O(s),O(s^{d-1})).
\end{equation}
If $r=n$, the same argument applies directly
without eliminating a transverse component.

We next estimate $(\nabla_{e_s}\Hess h)(e_s,e_s)$.
Although the full Hessian need not be small because of the
$B$ block, its contraction with $e_s$ is controlled by the
small eigenvalue $\theta_s$.

More precisely, the coordinate formula for the covariant
derivative shows that every term in
\[
(\nabla_{e_s}\Hess h)(e_s,e_s)
-D^3h(z_s)[e_s,e_s,e_s]
\]
contains either a first derivative of $h$ or a contraction
of $D^2h(e_s,\cdot)$ with a bounded vector.
Here the Christoffel symbols and their first derivatives
are bounded near the origin. Since
\(
\Hess h(e_s,\cdot)=\theta_sg_s(e_s,\cdot)\) and
\(D^2h=\Hess h+\Gamma\,dh,
\)
the difference is
$O(|\theta_s|+|dh|_g)=O(s^{d-2})$.
The quadratic $v$ term in \eqref{eq:splitting} has zero
third coordinate derivative. Moreover,
\eqref{eq:eigen-asymptotics} gives
$(e_s)_u=u_0+O(s)$.
Using the homogeneity of $P_d$ and the Taylor expansion
of $\rho$, we obtain
\begin{align*}
(\nabla_{e_s}\Hess h)(e_s,e_s)
&=D^3\rho(su_0)
  [(e_s)_u,(e_s)_u,(e_s)_u]+O(s^{d-2})\\
&=d(d-1)(d-2)a\,s^{d-3}+O(s^{d-2}).
\end{align*}
The leading coefficient is strictly positive.
On the other hand, \eqref{eq:MB-criterion} applies to
$e_s$ and bounds the absolute value of this expression
by $C|dh|_g=O(s^{d-1})$.
These estimates are incompatible as $s\searrow0$.

We conclude that $\rho\equiv0$ near the origin.
In the splitting coordinates \eqref{eq:splitting},
the critical set is therefore exactly $\{(u,v) \ | \ v=0\}$,
and its tangent space agrees with the kernel of
the Hessian at each of its points.
This is the Morse--Bott property.

\end{proof}

The preceding argument also applies to the weaker estimate
in which the right-hand side of \eqref{eq:MB-criterion}
is replaced by $C(|dh|_g+|\lambda|)$, where $\lambda$ is
the corresponding Hessian eigenvalue. In fact, this estimate
is also necessary and therefore provides a local
characterization of the Morse--Bott property, which may
be of independent interest.

\begin{theorem}\label{cor:MB-characterization}
Let $(N,g)$ be a real-analytic Riemannian manifold,
let $h:N\to\mathbb R$ be real-analytic, and let
$x_0$ be a critical point of $h$.
Denote by $H$ the self-adjoint endomorphism defined by
\[
g(HX,Y)=\operatorname{Hess}h(X,Y).
\]
Then $h$ is Morse--Bott near $x_0$ if and only if
there exist a neighborhood $U$ of $x_0$ and a constant
$C>0$ such that every unit eigenvector $e$ of $H$
corresponding to a simple eigenvalue $\lambda$,
at every point of $U$, satisfies
\begin{equation}\label{eq:morse-bott-characterization}
\left|(\nabla_e\operatorname{Hess}h)(e,e)\right|
\le C\bigl(|dh|_g+|\lambda|\bigr).
\end{equation}
\end{theorem}

\begin{proof}

Suppose first that \eqref{eq:morse-bott-characterization}
holds. The argument of Lemma \ref{lem:MB} still applies.
If the analytic remainder $\rho$ were nonzero, it would
produce simple Hessian eigenvalues
$\lambda_s=O(s^{d-2})$ and corresponding unit
eigenvectors $e_s$ such that
\[
|dh(z_s)|_g=O(s^{d-1})\ \ \text{ and } \ \ 
(\nabla_{e_s}\operatorname{Hess}h)(e_s,e_s)
=c\,s^{d-3}+O(s^{d-2}),
\]
with $c>0$. Thus the right-hand side of
\eqref{eq:morse-bott-characterization} is
$O(s^{d-2})$, contradicting its left-hand side.
Hence $\rho\equiv0$, and $h$ is Morse--Bott near $x_0$.

Conversely, suppose that $h$ is Morse--Bott near $x_0$.
By the Morse--Bott lemma, there are smooth local
coordinates $(u,v)$ centered at $x_0$ in which
\[
h=h(x_0)+\frac12v^{\mathsf T}Bv,
\]
where $B$ is constant, symmetric, and invertible.
In particular, $D^3h=0$ in these coordinates.

Writing the covariant Hessian in coordinates as
\[
(\operatorname{Hess}h)_{ij}
=\partial_i\partial_jh
-\Gamma^k_{ij}\partial_kh,
\]
and differentiating covariantly, we see that
$(\nabla_e\operatorname{Hess}h)(e,e)$
is a sum of terms involving either $dh$ or
contractions of $D^2h(e,\cdot)$ with the Christoffel
symbols. The Christoffel symbols and their first
derivatives are bounded on a sufficiently small
neighborhood of $x_0$.

If $e$ is a unit eigenvector of $H$ with eigenvalue
$\lambda$, then
\(
\operatorname{Hess}h(e,\cdot)
=\lambda g(e,\cdot).
\)
Since
\(
D^2h=\operatorname{Hess}h+\Gamma\,dh,
\)
it follows that
\[
|D^2h(e,\cdot)|
\le C_1\bigl(|\lambda|+|dh|_g\bigr)
\]
for some constant $C_1$, after shrinking the
neighborhood if necessary.

Therefore
\[
\left|(\nabla_e\operatorname{Hess}h)(e,e)\right|
\le C\bigl(|dh|_g+|\lambda|\bigr)
\]
for a uniform constant $C>0$.

In fact, this estimate holds for every unit
eigenvector of $H$, without requiring its eigenvalue
to be simple.
\end{proof}

Notice that the necessity of
\eqref{eq:morse-bott-characterization} requires only
smoothness, not analyticity.
We now return to Dupin hypersurfaces, which in fact
satisfy the stronger criterion \eqref{eq:MB-criterion}
of Lemma~\ref{lem:MB}.

\begin{proposition}\label{prop:distances}
For every immersion satisfying the hypotheses of
Theorem~\ref{thm:A}, all squared distance functions $L_q$ and
all nonzero linear height functions are Morse--Bott.
\end{proposition}
\begin{proof}
By Theorem~\ref{thm:A}, the image branches, induced metric,
and local shape operator are analytic. For $q \in \R^{m+1}$ and $h=L_q$, define $t=\langle q-f,\nu\rangle$ to be the normal component of the difference vector $q-f$ on the chosen neighborhood. Direct differentiation then gives
\begin{equation}\label{eq:distance-identities}
 H=I-tA\ \ \text{ and } \  dt(X)=dh(AX).
\end{equation}
If $e$ is a local unit principal vector for a simple
principal curvature $\lambda$, then
\begin{equation}\label{eq:distance-third}
 (\nabla_e\Hess h)(e,e) = e(1-t\lambda)-2\operatorname{Hess}h(\nabla_e e,e)
 =e(1-t\lambda)\stackrel{(\ref{eq:distance-identities})}{=}-\lambda^2dh(e)-t\,e(\lambda).
\end{equation}
The terms involving $\nabla_e e$ vanish since $e$ is
a unit eigenvector. Condition \eqref{eq:D} gives
$e(\lambda)=0$ on $\G$. But wherever $\lambda$ is simple,
$\lambda$ and a local unit eigenvector are smooth, so this
identity extends from $\G$ by continuity. 

At a critical point $p$ with $t=0$, i.e., $f(p) =q$, the Hessian is the
identity and the claim is true. Otherwise we can restrict ourselves to a small neighborhood in which
$t\ne0$ and $A$ is bounded. There $H$ and $A$ have
the same eigenspaces and multiplicities, and
\eqref{eq:distance-third} gives the criterion
\eqref{eq:MB-criterion}, since $\lambda$ is locally bounded. Lemma~\ref{lem:MB} then proves
the assertion for $L_q$.

For $h_a=\langle f,a\rangle$ with $a\ne0$, we put
$t=\langle\nu,a\rangle$ and obtain $H=tA$ and $dt(X)=-dh_a(AX)$ as well as
\[
 (\nabla_e\Hess h_a)(e,e)=-\lambda^2dh_a(e)+t\,e(\lambda)
\]
as the analogous identities. At a critical point, $t=\pm|a|\ne0$, so the same
argument applies.
\end{proof}

Proposition~\ref{prop:distances} also gives us the missing curvature surfaces at
singular points. In a local trivialization of the normal bundle the
normal exponential map is
\[
 \eta(x,t)=f(x)+t\nu(x)\ \ \text{ with } \ \
 d\eta_{(x,t)}(X,b)=df((I-tA)X)+b\nu.
\]
The fiber over $q$ identifies with $\Crit L_q$ via
$t=\langle q-f,\nu\rangle$. On this critical set $dt=0$
by \eqref{eq:distance-identities}, and Morse--Bott
regularity gives
\[
 T_{(x,t)}\eta^{-1}(q)
 =\ker d\eta_{(x,t)}
 =\{(X,0)\ | \ X\in\ker(I-tA)\}.
\]
For $t\ne0$, bundle projection to $M$ gives a curvature surface
tangent to the full principal space of curvature
$1/t$, which is constant along the surface. Conversely,
a curvature surface of constant nonzero curvature
$\lambda$ lifts to an integral manifold of
$\ker d\eta$ by $x\mapsto(x,1/\lambda)$.
This establishes the fiber integrability condition in
\cite[Theorem~2.9]{Wiesendorf}, including at singular
focal vectors, and explains its relation to curvature surfaces. For a zero principal curvature at $p$,
the critical manifold of $\langle f,\nu(p)\rangle$
is tangent to $\ker A$ and has locally constant normal.

With this identification, the tautness claim for compact embedded Dupin hypersurfaces follows directly from Proposition \ref{prop:distances} and the fiber integrability condition in \cite{Wiesendorf}.

\begin{proposition}\label{thm:B}
Let $f:M^m\to\R^{m+1}$ be a $C^\infty$ embedding of a compact
connected manifold satisfying \eqref{eq:D}. Then $f$ is
$\Z_2$-taut.
\end{proposition}

\begin{proof}
We set $L:=f(M)$. For $q\in\R^{m+1}$, the $H^1$-path
space of absolutely continuous paths with square-integrable derivative from $L$ to $q$ is identified with
$M\times H^1_0([0,1],\R^{m+1})$, where $H^1_0([0,1],\R^{m+1})$ denotes the space of paths $c\in H^1([0,1],\R^{m+1})$ with $c(0)=c(1)=0$, by
\[
 (x,c)\longmapsto\bigl(t\mapsto(1-t)f(x)+tq+c(t)\bigr).
\]
The energy functional on this space becomes
\begin{equation}\label{eq:energy-splitting}
 E_q(x,c)=L_q(x)+\tfrac12\int_0^1|c'(t)|^2\,dt,
\end{equation}
since the cross term is zero due to the boundary condition. The second summand is
a nondegenerate positive quadratic form. Thus
Proposition~\ref{prop:distances} makes every $E_q$
Morse--Bott. As $L$ is closed and embedded,
\cite[Theorem~2.9]{Wiesendorf} gives $\Z_2$-tautness.
The splitting preserves critical indices, and its
Hilbert space factor is contractible. Thus every Morse distance
function on $M$ is perfect.
\end{proof}

Note that the local Morse--Bott conclusion required no embeddedness. But the global step above uses the theorem for a closed embedded
submanifold \cite[Theorem~2.9]{Wiesendorf} and is not an assertion of tautness for an arbitrary
compact Dupin immersion. The same proof applies to a proper
embedding with tautness understood through homology of sublevel
sets.

\subsection{Passage to arbitrary codimension}
\label{sec:tubes}

\begin{proof}[Proof of Theorem~\ref{thm:general}]
We first check that a small tube satisfies the hypersurface
condition \eqref{eq:D}. This is the usual tube construction
for Dupin submanifolds \cite{Pinkall,CCJpre}, applied here only
on the regular locus. Consider a relatively compact
neighborhood $U\subset M$ and choose $\varepsilon>0$ so small
that $I-\varepsilon A_\xi$ is invertible for every unit normal
$\xi$ over $U$. Then
\[
 T_\varepsilon:\nu^1M|_U\longrightarrow\R^n,
 \qquad T_\varepsilon(p,\xi)=f(p)+\varepsilon\xi
\]
defines a smooth map on the unit normal bundle. Write a tangent
vector to its domain as $X^h+V$, where
$V\in\nu_pM\cap\xi^\perp$ is vertical, and put
$\nu_T(p,\xi)=\xi$. The Weingarten formula gives
\[
 \begin{aligned}
 dT_\varepsilon(X^h+V)
   &=df((I-\varepsilon A_\xi)X)+\varepsilon V,\\
 d\nu_T(X^h+V)
   &=-df(A_\xi X)+V.
 \end{aligned}
\]
The two summands in the first line are orthogonal. Since
$I-\varepsilon A_\xi$ is invertible, $dT_\varepsilon$ is
injective. Its image is perpendicular to $\xi$, so
$T_\varepsilon$ is a hypersurface immersion with unit normal
$\nu_T$.

Let $A_T$ denote its shape operator, characterized by
$dT_\varepsilon(A_TZ)=-d\nu_T(Z)$. If
$A_\xi X=\lambda X$, the two differential formulas show that
$X^h$ is an eigenvector of $A_T$ and the same holds for every
vertical vector $V$. Comparing the corresponding factors
gives the principal curvatures
\begin{equation}\label{eq:general-tube-curvatures}
 \mu_\lambda=\frac{\lambda}{1-\varepsilon\lambda}\ \
 \text{ and } \ \ \mu_{\mathrm{vert}}=-\frac1\varepsilon,
\end{equation}
where $\mu_{\mathrm{vert}}$ occurs only if $n-m\ge2$, with
multiplicity $n-m-1$. The eigenspace of $\mu_\lambda$ is the
horizontal lift of $\ker(A_\xi-\lambda I)$, and that of
$\mu_{\mathrm{vert}}$ is the vertical tangent space.
The map $\lambda\mapsto\mu_\lambda$ is injective
and never takes the value $\mu_{\mathrm{vert}}$.
Thus the number of distinct principal curvatures is unchanged
in codimension one and increases by exactly one in higher
codimension.
Consequently the regular
locus of the tube, viewed in its parameter space \(\nu^1M|_U\), is exactly the restriction of $\G_\nu$.
On this locus, for $X\in\ker(A_\xi-\lambda I)$,
\[
 d\mu_\lambda(X^h)
 =\frac{d\lambda(X^h)}{(1-\varepsilon\lambda)^2}=0,
\]
while $\mu_{\mathrm{vert}}$ is constant.
Hence \eqref{eq:D} holds on the tube.

Theorem~\ref{thm:A} makes every local image branch of
$T_\varepsilon$ Nash. Its unit normal $\nu_T$ is Nash as well,
so the map
\[
 z\longmapsto z-\varepsilon\nu_T(z)
\]
is Nash. Its composition with $T_\varepsilon$ is $f\circ\pi$,
where $\pi:\nu M|_U\to U$ is the bundle projection,
restricted to $\nu^1M|_U$ in this composition.
Since $\pi$
is a submersion and $f$ is an immersion, this composition
has constant rank $m$. On each local image branch,
$T_\varepsilon$ is a local diffeomorphism, so the Nash map
above also has constant rank $m$. Its local image is precisely
a local image branch of $f$.
The Nash constant rank theorem therefore shows that every local
image branch of $f$ is Nash. These image charts also give the
claimed compatible analytic structure on $M$.

To transfer the Morse--Bott conclusion, denote the normal
endpoint map of $M$ by
\[
 \eta_M(p,v)=f(p)+v.
\]
For the tube its counterpart is
\[
 \eta_T(p,\xi,t)=f(p)+(\varepsilon+t)\xi.
\]
On $\varepsilon+t>0$, the map
$\Psi(p,\xi,t)=(p,(\varepsilon+t)\xi)$ is a diffeomorphism
onto the punctured normal bundle over $U$, and
\(
 \eta_T=\eta_M\circ\Psi
\).
Consequently, $\Psi$ identifies the corresponding local
fibers, and differentiation gives
\(
 d\Psi\bigl(\ker d\eta_T\bigr)=\ker d\eta_M\).
Thus both smoothness of the local fibers and equality of
their tangent spaces with the full differential kernels
are preserved under this change of variables.
By Proposition~\ref{prop:distances} and the discussion following
it, the local fibers of $\eta_T$ are smooth and tangent
to the full kernels of its differential. The same is therefore
true for $\eta_M$ away from the zero section. On the zero
section $d\eta_M$ is invertible.

For a fixed $q\in\R^n$, projection to $M$ identifies
$\eta_M^{-1}(q)$ with $\Crit L_q$, since necessarily
$v=q-f(p)$. The endpoint differential formula recalled
in the introduction shows that $\ker d\eta_M$ is horizontal
and projects isomorphically onto $\ker(I-A_v)$.
The projection therefore restricts locally to an embedding
of each smooth fiber, so $\Crit L_q$ is locally a submanifold.
At a critical point $p$, the Hessian endomorphism of $L_q$
is $I-A_v$, and hence
\[
 T_p\Crit L_q
 =d\pi_{(p,v)}\bigl(\ker d\eta_{M,(p,v)}\bigr)
 =\ker\Hess_p L_q.
\]
This proves (1), including at focal vectors where the
multiplicities change.

If $M$ is compact and embedded, the energy splitting
\eqref{eq:energy-splitting}, now with values in $\R^n$, and
\cite[Theorem~2.9]{Wiesendorf} prove (2).
\end{proof}

\begin{proof}[Proof of Corollary~\ref{cor:characterization}]
The forward implication follows from Theorem~\ref{thm:general}.
For the converse, suppose that every $L_q$ is Morse--Bott.
Fix $(p,\xi)\in\G_\nu$, a local eigenvalue function $\lambda$
with $\lambda(p,\xi)\ne0$, and a unit vector
$e\in\ker(A_\xi-\lambda(p,\xi)I)$. Put
$t=\lambda(p,\xi)^{-1}$ and $q=f(p)+t\xi$. Then $p$ is a
critical point of $L_q$, and its Hessian endomorphism is
$I-tA_\xi$. Thus $e\in\ker\Hess_pL_q$.

The Morse--Bott condition provides a critical submanifold
through $p$ tangent to this kernel. Differentiating along a
curve in this submanifold with initial tangent $e$ gives
$ (\nabla_e\Hess L_q)_p(e,e)=0$, where $\nabla$ is the
Levi-Civita connection of the induced metric. Write
$v=(q-f)^\perp$ for the normal component of $q-f$.
At $p$ one has $v=t\xi$ and $\nabla_e^\perp v=0$.
Differentiating the distance Hessian formula $H=I-A_v$,
with $\xi$ extended normally parallel along a curve tangent
to $e$, therefore yields
\[
 0=(\nabla_e\Hess L_q)_p(e,e)
   =-t\,d\lambda_{(p,\xi)}(e^h).
\]
Hence \eqref{eq:Dnu} holds wherever $\lambda\ne0$.
It also holds on the interior of the zero set of $\lambda$,
and continuity gives it at the remaining points.

In the compact embedded case, tautness implies the
Morse--Bott property by \cite{Ozawa}, while
Theorem~\ref{thm:general} gives the converse.
\end{proof}

\section{Consequences of principal multiplicities}\label{sec:multiplicity}

On an open set where a principal curvature $\lambda$ has constant
multiplicity at least two, one has $d\lambda|_{E_\lambda}=0$.
To see this, choose local orthonormal fields $X,Y\in E_\lambda$. The
Codazzi equation and self-adjointness then give
\[
 X(\lambda)
 =\langle(\nabla_XA)Y,Y\rangle
 =\langle(\nabla_YA)X,Y\rangle=0,
\]
because
$(\nabla_YA)X=Y(\lambda)X+(\lambda I-A)\nabla_YX$. The last equality follows from $X\perp Y$ and
$(\lambda I-A)Y=0$. Every unit vector in $E_\lambda$
can be used as $X$ at the chosen point.

This standard consequence of Codazzi is recalled, for example, in
\cite[Section~3, p.~20]{Cecil}. Thus the only nonautomatic part of
\eqref{eq:D} is its assertion for simple principal curvatures on
the regular locus.

\begin{proof}[Proof of Corollary~\ref{cor:analytic}]
As discussed above, the Codazzi equation gives \eqref{eq:D} on $\G$,
so Theorem~\ref{thm:A} applies.
\end{proof}

A compact connected taut immersion into Euclidean space is an embedding \cite[Theorem 2.5]{TT}. To apply Proposition~\ref{thm:B} to derive Corollary~\ref{cor:taut}, we need embeddedness before tautness has been established. The following proposition derives it directly from the multiplicity assumption. Its proof uses the same uniqueness-of-minima argument as Terng and Thorbergsson, with the absence of index-one critical points providing the required uniqueness.

\begin{proposition}\label{prop:embedding}
A compact connected immersed hypersurface in Euclidean space with
no simple principal curvature on its regular locus is embedded.
\end{proposition}
\begin{proof}
A simple principal curvature at any point would remain
simple on a neighborhood meeting $\G$, since multiplicity in singular points can only increase. Thus no
principal curvature is simple anywhere. At a critical
point of $L_q$, its Hessian is $I-tA$, so its index
is a sum of principal multiplicities. In particular,
a Morse distance function has no critical point of
index one. Such Morse functions occur for dense
centers $q$, by Sard's theorem for the normal
endpoint map.

A Morse function on a compact connected manifold
with no index-one critical points has exactly one
minimum, because minima create components, and handles of
index at least two have connected attaching regions
and cannot join them. If $f(p_1)=f(p_2)=q$ for
$p_1\ne p_2$, both points are nondegenerate minima
of $L_q$, with Hessian $I$. They persist in disjoint
neighborhoods for nearby centers. Choosing a nearby
Morse center contradicts uniqueness of the minimum.
Hence $f$ is injective, and compactness makes it
an embedding.
\end{proof}

\begin{proof}[Proof of Corollary~\ref{cor:taut}]
Proposition~\ref{prop:embedding} gives embeddedness,
while again the Codazzi equation gives \eqref{eq:D}.
Proposition~\ref{thm:B} now applies and yields the claim.
\end{proof}

\begin{remark}\label{rem:divisor}
For a compact connected immersion, if all regular principal multiplicities share a divisor
$d\ge2$, this divisibility persists at singular points,
where multiplicities add. Every Morse distance index
is divisible by $d$, so the Morse complex has no
nonzero differentials and the immersion is taut.
Chi's theorem then gives analyticity in the compact
embedded case \cite{Chi}. The argument does not cover
the weaker assumption that all multiplicities are at
least two, because sums of such multiplicities may be
consecutive in this case. Our analyticity result also requires
neither compactness nor embeddedness.
\end{remark}

\section{Finiteness and deformations}\label{fin:sec:carriers}

The basic idea of the finiteness argument is to
describe the submanifolds under consideration by
semialgebraic families depending on finitely many
real parameters. The uniform degree bound obtained
in Theorem~\ref{thm:A} provides the starting point
for controlling the complexity of these descriptions
independently of the individual submanifold.
Hardt's semialgebraic triviality theorem provides
the necessary control of connected components
in such families. Once a uniform bound on description
complexity has been established, the submanifolds
can be organized into finitely many semialgebraic
parameter families.

The Nash triviality theorem of Coste--Shiota
then provides finite partitions of the parameter
spaces over which these families admit Nash
trivializations. Suitable Nash paths within
these parameter pieces, together with the
trivializations, yield Nash isotopies between
the corresponding embedded submanifolds.
For the stronger finiteness results, we further
show that tautness defines a semialgebraic
condition on the parameters. Restricting the
families to the taut parameter sets then allows
the isotopies to preserve tautness.

The algebraic carrier provided by
Theorem~\ref{thm:A} may have singularities and
additional real branches. We therefore first have to
pass from this carrier to the embedded
hypersurface itself.

Throughout the remaining sections, tautness and
homology are understood over $\Z_2$. We write
\[
 b_j(M)=\dim_{\Z_2}H_j(M;\Z_2) \ \ \text{ and } \ \  \beta(M)=\sum_jb_j(M).
\]

\subsection{From degree to description complexity}
For a polynomial $P$ in $n$ variables, write
$Z(P)=\{x\in \R^n \ | \ P(x)=0\}$ and
$Z_{\mathrm{reg}}(P)=\{x\in \R^n \ | \ P(x)=0\ \wedge\ |(\grad P)(x)|^2>0\}$.

\begin{lemma}[Selecting regular components]\label{fin:lem:selection}
Let $H\subset\R^n$ be a compact smooth embedded hypersurface, and
let $P\in\R[x_1,\ldots,x_n]$ be a nonconstant irreducible
polynomial with $H\subset Z(P)$. Then $H$ is the closure of a
union of whole connected components of $Z_{\mathrm{reg}}(P)$.
In particular, $H$ is semialgebraic and is a Nash submanifold of
$\R^n$.
\end{lemma}

\begin{proof}
If a polynomial $Q$ vanishes on an open piece of $H$,
then $P\mid Q$. Indeed, choose affine graph coordinates
$(u,t)$ on that piece. The polynomial $P$ depends on
$t$, since otherwise it vanishes on an open subset of
the $u$-space. As in the proof of Lemma~\ref{lem:contact},
Gauss's lemma shows that $\Res_t(P,Q)$ is a nonzero
polynomial in $u$ whenever $P\nmid Q$.
Substituting the graph into the resultant identity makes
this polynomial vanish on an open set, which contradicts
$\Res_t(P,Q)\not\equiv0$.

It follows that $H\cap Z_{\mathrm{reg}}(P)$
is dense in $H$. Indeed, otherwise $\grad P$
would vanish on a nonempty open subset of $H$.
The preceding argument would then imply
$P\mid\partial_iP$ for every $i$.
Since $P$ is nonconstant, at least one of these
derivatives is nonzero and has degree smaller
than $\deg P$, a contradiction.
 At regular algebraic points, the
inclusion of $H$ into $Z_{\mathrm{reg}}(P)$ is a local
diffeomorphism between manifolds of dimension $n-1$.
Thus their intersection is open in the regular locus, and
it is also closed there because $H$ is compact. It is
therefore a union of whole connected components, and
\[
 H=\overline{H\cap Z_{\mathrm{reg}}(P)}.
\]
Since $Z_{\mathrm{reg}}(P)$ is semialgebraic,
it has finitely many connected components,
each of which is semialgebraic.
As $H\cap Z_{\mathrm{reg}}(P)$ is a union of
such components, it is semialgebraic.
Its closure is $H$, which is therefore
semialgebraic as well. The last assertion then follows from the fact that a smooth
semialgebraic submanifold is Nash \cite[Chapter~8]{BCR}.
\end{proof}

In the following we measure the description complexity of a semialgebraic
set $S\subset\mathbb R^n$ by the number and degrees of
the polynomials needed to describe it. More precisely,
we consider Boolean combinations of polynomial
equalities and strict inequalities
\[
P_i(x)=0\ \ \text{ or }\ \ P_i(x)>0,
\]
and say that $S$ has description complexity bounded
by $(s,d)$ if it admits such a description involving
at most $s$ polynomials, each of degree at most $d$.
No bounds on their coefficients are imposed.
For our purposes, this notion of uniformly bounded
complexity is equivalent to the degree convention
used by Coste--Shiota \cite{CS}.

\begin{lemma}[Uniform description complexity]\label{fin:lem:complexity}
For fixed integers $n$ and $D$, the closure of any union of
connected components of $Z_{\mathrm{reg}}(P)$, where $P$ is a
nonzero real polynomial of degree at most $D$, has a semialgebraic
description of complexity at most $c(n,D)$.
\end{lemma}

\begin{proof}
Parametrize all nonzero polynomials of degree at most $D$ by
their coefficient vectors $a\in T=\R^N\setminus\{0\}$,
where $N=\binom{n+D}{D}$, and apply Hardt triviality
\cite{Hardt} to
\[
 \mathcal R=\{(x,a)\ | \ P_a(x)=0 \wedge |\grad_xP_a(x)|^2>0\}.
\]
By Hardt's semialgebraic triviality theorem
\cite{Hardt}, the parameter space $T$ admits a
finite semialgebraic partition into sets $T_i$
such that the family is semialgebraically trivial
over each $T_i$, with a fixed fiber $F_i$. This means that
the family over $T_i$ is semialgebraically
homeomorphic to $F_i\times T_i$, with the
homeomorphism preserving the projection onto $T_i$.
In particular, any two fibers over points of $T_i$
are semialgebraically homeomorphic and the trivialization identifies the connected
components of each fiber with those of $F_i$.
Since $F_i$ has only finitely many connected
components, all possible unions of such components
are represented by finitely many semialgebraic
families $\mathcal U_{i,I}$ over $T_i$.

Their fiberwise closures are defined by
\[
(x,a)\in\mathcal K_{i,I}
\quad\Longleftrightarrow\quad
a\in T_i\ \wedge\
\bigl(\forall\delta>0\ \exists z:\
(z,a)\in\mathcal U_{i,I}\ \wedge\
|x-z|^2<\delta^2\bigr).
\]
In this formula, the parameter $a$ is kept fixed
throughout the quantification. Thus the condition
describes the closure of the selected union of
components within each individual fiber, rather
than the closure of the total family in the
product space.

Since the family $\mathcal U_{i,I}$ is semialgebraic,
again the Tarski--Seidenberg theorem, applied through
quantifier elimination, shows that $\mathcal K_{i,I}$
is semialgebraic as well. Moreover, the defining
formula involves only a fixed finite collection
of polynomial expressions in the variables
$(x,a)$, with degrees and numbers determined
by the family.

Quantifier elimination therefore provides a
semialgebraic description of $\mathcal K_{i,I}$ whose
complexity is bounded independently of $a$.
Since only finitely many families $\mathcal K_{i,I}$
occur, the maximum of these bounds depends
only on $n$ and $D$. Specializing the parameter
$a$ then yields a description of each individual
fiber closure with the same uniform complexity
bound.

\end{proof}

Recall that two compact Nash submanifolds of
$\R^n$ are Nash isotopic if there exists a
Nash family of embeddings
\[
\iota_t:M\longrightarrow\R^n,
\qquad 0\le t\le1,
\]
whose endpoint images are the given submanifolds.
Here the map $(x,t)\mapsto\iota_t(x)$ is required
to be jointly Nash.

Two embedded submanifolds have the same smooth
ambient isotopy type if one can be carried onto
the other by a smooth isotopy of diffeomorphisms
of $\R^n$ starting at the identity.

For compact embedded manifolds, a smooth isotopy
of embeddings admits a smooth ambient extension.
Indeed, after a smooth reparametrization of time
that is constant near the endpoints, prescribe
the velocity
\[
V_t(\iota_t(x))=\partial_t\iota_t(x)
\]
along the track of the isotopy. This velocity
field extends to a smooth time-dependent vector
field on $\R^n$ with support in a fixed compact
set. Its flow $\Phi_t$ exists for $0\le t\le1$
and satisfies
\[
\Phi_t\circ\iota_0=\iota_t
\]
by uniqueness. Consequently, Nash isotopy generally implies
smooth ambient isotopy for compact embedded
submanifolds.
\par\vspace{\baselineskip}

Combining the uniform description complexity established
above with the Nash finiteness theorem of Coste--Shiota
and the isotopy extension argument, we obtain the
following finiteness result.

\begin{theorem}[Finiteness for bounded-degree carriers]
\label{fin:thm:carrier}
For fixed integers $n\ge2$ and $D\ge1$, consider the
class of compact connected smooth embedded hypersurfaces
$H\subset\R^n$ whose images are contained in the zero
set of a nonconstant irreducible polynomial of degree
at most $D$. This class contains only finitely many
Nash isotopy types, smooth ambient isotopy types,
and abstract diffeomorphism types.
\end{theorem}

\begin{proof}
The two preceding lemmas make every such hypersurface Nash with
description complexity bounded by a constant $c(n,D)$.
Coste--Shiota's finiteness theorem
\cite[Theorem~B]{CS} then gives finitely many
Nash isotopy types. By the isotopy extension
argument recalled above, this also gives only
finitely many smooth ambient isotopy types
and abstract diffeomorphism types.
\end{proof}

\begin{corollary}[Ordinary finiteness for Dupin hypersurfaces]
\label{fin:cor:dupin}
For each $m\ge1$, the class of compact connected smooth
embedded hypersurfaces $M^m\subset\R^{m+1}$ satisfying
\eqref{eq:D} has uniformly bounded semialgebraic
description complexity and contains only finitely many
smooth ambient isotopy types and diffeomorphism types.
Moreover, there exists a constant $B(m)$ depending
only on $m$ such that
\(
\beta(M)\le B(m)
\)
for every hypersurface in this class.
\end{corollary}

\begin{proof}
By Theorem \ref{thm:A}, every hypersurface in the class
under consideration is contained in the zero set
of a nonconstant irreducible polynomial of degree
at most $D(m)$, where $D(m)$ depends only on $m$.
Lemmas \ref{fin:lem:selection} and \ref{fin:lem:complexity} therefore provide a uniform
bound on its semialgebraic description complexity.

Applying Theorem \ref{fin:thm:carrier} with $n=m+1$ and $D=D(m)$
shows that only finitely many smooth ambient
isotopy types and diffeomorphism types occur.
Since diffeomorphic manifolds have the same
Betti numbers, the total Betti number $\beta(M)$
is bounded by the maximum over a finite list
of representatives of these diffeomorphism types.
This gives the required bound $B(m)$.
\end{proof}

This argument uses Theorem~\ref{thm:A} alone and does not
yet retain the Dupin condition during the isotopy. It also
applies to the compact immersions in
Corollary~\ref{cor:taut}, which are embedded by
Proposition~\ref{prop:embedding}. Arbitrary compact Dupin
immersions need not be embeddings, as nontrivial finite
coverings of a Dupin torus show.

\subsection{Taut submanifolds in arbitrary codimension}
\label{fin:sec:taut}

\begin{corollary}[Fixed ambient dimension]\label{fin:cor:taut}
For each $n\ge2$, the class of compact connected
smooth taut submanifolds of $\R^n$, of arbitrary
codimension, has uniformly bounded semialgebraic
description complexity and contains only finitely
many smooth ambient isotopy types.

\end{corollary}

\begin{proof}
For hypersurfaces, tautness implies the Dupin condition
\cite{Miyaoka,Pinkall}, so Corollary~\ref{fin:cor:dupin} applies.
Let $N\subset\R^n$ have codimension $r\ge2$, and choose
$\varepsilon>0$ smaller than a tubular radius. The tube
\[
 H_\varepsilon=\{x\in \R^n\ | \ \dist(x,N)=\varepsilon\}
\]
is a compact embedded taut hypersurface by the tube
theorem \cite{Pinkall}. It is connected because its
normal sphere fibers $S^{r-1}$ and its base are
connected. Its description complexity is therefore
bounded solely in terms of $n$.

The open normal disk bundle
$\Omega_\varepsilon=\{x \in \R^n\ | \ \dist(x,N)<\varepsilon\}$
is the unique bounded component of
$\R^n\setminus H_\varepsilon$. 
The bounded component $\Omega_\varepsilon$ also has
uniformly bounded semialgebraic description complexity.
To see this, consider the finitely many possible
description formats for $H_\varepsilon$, with the
polynomial coefficients treated as parameters.

Applying Hardt's semialgebraic triviality theorem
to the complements of the hypersurfaces in these
families gives finite parameter partitions on which
the connected components of the complements can be
identified uniformly.

For a component $C_a$ selected in such a family,
boundedness can be expressed by the first-order
condition
\[
\exists R>0\ \forall x\in\R^n:\
\bigl(x\in C_a \Longrightarrow |x|^2<R^2\bigr).
\]
Since $\Omega_\varepsilon$ is the unique bounded
component of $\R^n\setminus H_\varepsilon$, this
condition identifies it among the labeled
components.

Quantifier elimination then provides a
semialgebraic description of the selected
component, with complexity bounded independently
of the coefficients of $H_\varepsilon$.
Since only finitely many description formats
and parameter pieces occur, the resulting bound
depends only on $n$.

The original submanifold is recovered by
\begin{equation}\label{fin:eq:recover}
N=\{x\in\Omega_\varepsilon \ | \
\forall y\in H_\varepsilon:\
|x-y|^2\ge\varepsilon^2\}.
\end{equation}
To verify this claim, first let $x\in N$.
Since every point of $H_\varepsilon$ has distance
$\varepsilon$ from $N$, we have
$|x-y|\ge\varepsilon$ for all $y\in H_\varepsilon$.
Thus $x$ belongs to the set on the right-hand side
of \eqref{fin:eq:recover}.
Conversely, suppose that $x\in\Omega_\varepsilon$
but $x\notin N$. Since $\varepsilon$ is smaller
than the tubular radius, there is a unique
representation $x=p+v$, where $p\in N$ and
$v\in\nu_pN$ satisfies $0<|v|<\varepsilon$.
The point
\[
y=p+\varepsilon\frac{v}{|v|}
\]
then belongs to $H_\varepsilon$, and
\(
|x-y|=\varepsilon-|v|<\varepsilon.
\)
Hence $x$ does not satisfy the condition on the
right-hand side of \eqref{fin:eq:recover},
which proves the equality.

We now regard $\varepsilon$ as an additional real
parameter in \eqref{fin:eq:recover}. Since
$\Omega_\varepsilon$ and $H_\varepsilon$ have
uniformly bounded semialgebraic description
complexity, quantifier elimination gives a
semialgebraic description of $N$ whose complexity
is bounded solely in terms of $n$.
In particular, the bound is independent of the
individual submanifold and of the chosen
tubular radius. No uniform positive lower bound
on this radius is needed.

Since $N$ is smooth and semialgebraic, it is a
Nash submanifold. The finiteness theorem of
Coste--Shiota, together with the isotopy extension
argument recalled above, therefore yields only
finitely many smooth ambient isotopy types.
The zero-dimensional case is included, since
every compact connected zero-dimensional
submanifold consists of a single point.

\end{proof}

Combining the preceding finiteness result for fixed
ambient dimension with Kuiper's bound on the
ambient dimension of substantial taut embeddings,
we obtain a corresponding finiteness statement
for fixed intrinsic dimension.

\begin{corollary}[Fixed intrinsic dimension]\label{fin:cor:intrinsic}
For each $k\ge1$, only finitely many diffeomorphism types of
compact connected $k$-manifolds admit a taut embedding in some
Euclidean space. Their total $\Z_2$-Betti numbers are bounded
by a constant $B_{\mathrm{taut}}(k)$ depending only on $k$.
\end{corollary}

\begin{proof}
Replace the ambient space of a taut embedding
$f:N^k\to\R^N$ by its affine span $E\cong\R^n$.
Tautness is unchanged, since for $q=q_E+q_\perp$
the squared distance splits as
\[
 \tfrac12|f(x)-q|^2
 =\tfrac12|f(x)-q_E|^2+\tfrac12|q_\perp|^2.
\]
The resulting embedding is substantial and has the
two-piece property. Kuiper's bound
\cite{Kuiper}, recalled in \cite[Theorem~1.1]{Cecil},
gives $n\le k(k+3)/2$. Only finitely many ambient
dimensions remain, and
Corollary~\ref{fin:cor:taut} gives a finite list of
diffeomorphism types for each. Taking their union
and the maximal total Betti number proves the claim.
\end{proof}

\subsection{Retaining tautness in a family}
\label{fin:sec:deformations}

The preceding finiteness results do not yet ensure
that tautness is preserved along the isotopies.
To obtain this stronger conclusion, we must
restrict the parameter families to those fibers
that are taut. The key observation is that
tautness itself defines a semialgebraic condition
on the parameters, as established in the following
lemma. This allows us to apply Nash triviality
within the taut parameter sets and construct
isotopies whose intermediate submanifolds
remain taut.

\begin{lemma}[Tautness in a semialgebraic family]
\label{fin:lem:tautfamily}
Let $B$ be a semialgebraic set and let
$\mathcal X\subset\R^n\times B$ be semialgebraic. Suppose
that each fiber
\[
 X_b=\{x\in\R^n\ | \ (x,b)\in\mathcal X\}
\]
is a nonempty compact connected Nash submanifold without boundary.
Then the parameter set
\[
 T=\{b\in B\ | \ X_b\text{ is taut}\}
\]
is semialgebraic. Moreover, finitely many of the fibers with
parameter in $T$ suffice as representatives for Nash isotopy
through taut fibers of this family.
\end{lemma}

\begin{proof}
By the Nash triviality theorem of Coste--Shiota
\cite[Theorem~A]{CS}, the parameter space $B$ admits
a finite partition into Nash manifolds $B_i$ such
that the family is Nash trivial over each $B_i$.
Thus, for every $i$, there exists a compact Nash
manifold $F_i$ and a Nash trivialization
\[
(x,b)\longmapsto (h_i(x,b),b)
\ \ \text{ for }\ \  (x,b)\in F_i\times B_i,
\]
where $h_{i,b}=h_i(\cdot,b)$ is a Nash embedding
of $F_i$ onto $X_b$.

The advantage of this representation is that all
fibers over $B_i$ are parametrized by the same
manifold $F_i$. We can therefore study their
squared distance functions on a fixed domain.
Put $\beta_i=\beta(F_i)$ and consider $L_{b,q}: F_i \to \R$, defined by
\[
L_{b,q}(x)=\frac12|h_i(x,b)-q|^2.
\]

Since $h_i$ is Nash, the first and second
derivatives of $L_{b,q}$ in a finite Nash atlas
of $F_i$ depend semialgebraically on $(x,b,q)$.
Consequently, both the set of critical points
and the subset of degenerate critical points
are semialgebraic.

By projecting the latter set onto the parameter
space $B_i\times\R^n$, the Tarski--Seidenberg
theorem shows that the set of pairs $(b,q)$
for which $L_{b,q}$ has a degenerate critical
point is semialgebraic. Its complement
$\mathcal M_i$ therefore consists precisely
of those pairs for which $L_{b,q}$ is Morse
and is also semialgebraic.

We now characterize tautness in terms of these
parameters. For a Morse function on the compact
manifold $F_i$, the weak Morse inequalities give
\(
c_j(L_{b,q})\ge b_j(F_i)
\)
in every degree, where $c_j(L_{b,q})$ denotes
the number of critical points of index $j$.
In particular,
\(
\#\Crit L_{b,q}\ge\beta_i.
\)
Equality holds precisely when the Morse function
is perfect, since equality of the total numbers
forces equality in every degree.

Moreover, the condition
\(
\#\Crit L_{b,q}=\beta_i
\)
can be expressed by a first-order formula.
Indeed, it requires the existence of exactly
$\beta_i$ distinct critical points in $F_i$,
with every critical point equal to one of them.
All critical points are counted on the fixed
manifold $F_i$, independently of the choice
of local charts.

It follows that the parameters corresponding
to taut fibers are exactly those satisfying the condition
\begin{equation}\label{fin:eq:tautformula}
\forall q\in\R^n:\
\bigl((b,q)\in\mathcal M_i
\Longrightarrow \#\Crit L_{b,q}=\beta_i\bigr).
\end{equation}
Since all conditions in this formula are
semialgebraic, quantifier elimination shows
that $T\cap B_i$ is semialgebraic.
As the partition of $B$ is finite, the set
$T$ of parameters corresponding to taut fibers
is semialgebraic as well.

We now use this result to obtain the claimed
finiteness up to Nash isotopy through taut
submanifolds. Following the parameter-path
construction of Coste--Shiota
\cite[Section~5]{CS}, we refine each
semialgebraic set $T\cap B_i$ into finitely
many Nash cells $C_{i,\alpha}$.

Each cell $C_{i,\alpha}$ admits a Nash diffeomorphism
onto an open box. Given any parameter $b\in C_{i,\alpha}$,
we can join it to a fixed reference parameter
$b_{i,\alpha}\in C_{i,\alpha}$ by taking the inverse
image of a straight line segment in the corresponding
box. This gives a Nash path
\(
\gamma:[0,1]\longrightarrow C_{i,\alpha}
\)
joining $b$ to $b_{i,\alpha}$.

Since the family is Nash trivial over $B_i$,
the maps
\(
h_{i,\gamma(t)}:F_i\longrightarrow\R^n
\)
form a Nash isotopy of embeddings.
Every intermediate image is taut because
$\gamma(t)\in T$ for all $t\in[0,1]$.
Composing with the inverse of the initial
parametrization makes the isotopy start
from the inclusion of $X_b$.

Finally, choosing one reference fiber for
each of the finitely many Nash cells gives
a finite list of representatives, such that
every taut fiber is Nash isotopic to one
of these representatives through taut
embedded submanifolds.
\end{proof}

We now have all the ingredients needed to prove
Theorem \ref{fin:thm:tautdeformations}. The uniform description complexity
established above allows us to organize the
submanifolds into finitely many semialgebraic
families. Lemma \ref{fin:lem:tautfamily} then ensures that these
families admit finite representative lists
up to Nash isotopy through taut submanifolds.

\begin{proof}[Proof of Theorem~\ref{fin:thm:tautdeformations}]
By Corollary~\ref{fin:cor:taut}, all submanifolds
under consideration admit semialgebraic descriptions
of uniformly bounded complexity. Consequently,
they can be represented by finitely many polynomial
description formats. Treating the polynomial
coefficients as parameters, each such format
defines a semialgebraic family.

We restrict these families to parameters for which
the corresponding fibers are nonempty, compact,
connected Nash submanifolds. This restriction is
again semialgebraic. Indeed, closedness and
boundedness can be expressed by first-order
conditions, while connectedness can be controlled
by Hardt's semialgebraic triviality theorem.
For a fixed description format, the condition
that a fiber is a Nash submanifold is also
semialgebraic. This follows from the finite
differentiability bounds for semialgebraic
graphing functions discussed in
\cite[Section~2, Proposition~2.1 and the
discussion of Theorem~2.4]{CS}.
If necessary, we further partition the parameter
spaces according to the dimensions of the fibers.

We can now apply Lemma~\ref{fin:lem:tautfamily}
to each of these finitely many families.
It provides a finite list of representative
fibers such that every taut submanifold under
consideration is Nash isotopic to one of them
through taut embedded submanifolds.
Combining the representative lists gives the
claimed finiteness up to Nash isotopy.
The smooth ambient extensions follow from
the isotopy extension argument recalled before
Theorem~\ref{fin:thm:carrier}. Consequently,
only finitely many smooth ambient isotopy types
and abstract diffeomorphism types occur.

Finally, the total $\Z_2$-Betti number is constant
under diffeomorphism. Taking its maximum over
the finite representative list therefore gives
a uniform bound depending only on $n$.
\end{proof}

We now apply Theorem~\ref{fin:thm:tautdeformations}
to compact embedded Dupin
hypersurfaces. Since these hypersurfaces are taut
by Proposition \ref{thm:B}, and tautness
implies the Dupin condition, the resulting Nash
isotopies remain within the class of Dupin
hypersurfaces. This yields the following
specialization of our finiteness theorem.

\begin{corollary}\label{thm:C}
For each $m\ge1$, there are finitely many compact connected
embedded Nash Dupin hypersurfaces
$H_1,\ldots,H_{N(m)}\subset\R^{m+1}$ such that every compact
connected smooth embedded hypersurface $M^m\subset\R^{m+1}$
satisfying \eqref{eq:D} is joined to one of the $H_i$ by a Nash
isotopy through embedded Dupin hypersurfaces. This isotopy admits
a smooth ambient extension,  possibly after a smooth
reparametrization of time.

Consequently only finitely many smooth ambient isotopy types and
diffeomorphism types occur. Moreover, there exists a
constant $B(m)$ depending only on $m$ such that
\[
\beta(M)\le B(m)
\]
for every hypersurface in this class.
\end{corollary}

\begin{proof}[Proof of Corollary~\ref{thm:C}]
By Proposition~\ref{thm:B}, every compact connected
embedded hypersurface satisfying \eqref{eq:D} is taut.
We can therefore apply
Theorem~\ref{fin:thm:tautdeformations} in
$\R^{m+1}$ and restrict its finite representative
list to submanifolds of dimension $m$, since
dimension is preserved under Nash isotopy.

The resulting isotopies consist entirely of
compact embedded taut hypersurfaces.
Since taut hypersurfaces satisfy the Dupin
condition \cite{Miyaoka,Pinkall}, every
intermediate hypersurface is Dupin as well.
Thus the isotopies preserve the required
geometric condition.

The smooth ambient extensions follow from
Theorem~\ref{fin:thm:tautdeformations}.
The finiteness of ambient isotopy and
diffeomorphism types follows from the finite
representative list, which also gives a
uniform bound on $\beta(M)$ depending only
on $m$.
\end{proof}

Although the isotopies preserve tautness and hence
the Dupin condition, the number of distinct principal
curvatures and their multiplicities need not remain
constant along the deformation.

\section{Topological consequences}\label{fin:sec:topology}

Theorem~\ref{thm:general}, together with classical classification
and rigidity results for taut embeddings, yields several
topological consequences of the regular Dupin condition.
None of the following applications assumes tautness or constant
principal multiplicities in advance. They are independent of
the finiteness results of Section~\ref{fin:sec:carriers}.

\begin{corollary}[Homology spheres]\label{fin:cor:homologyspheres}
Let $M^m\subset\R^n$, $1\le m<n$, be a compact connected smooth
embedded submanifold satisfying \eqref{eq:Dnu}.
If $M$ has the $\Z_2$-homology of $S^m$, then $M$ is a round
sphere in an affine $(m+1)$-dimensional subspace of $\R^n$.
\end{corollary}

\begin{proof}
By Theorem~\ref{thm:general}, $M$ is taut. Every Morse squared
distance function therefore has exactly one critical point of
index $0$ and one of index $m$, and no others. For $m\ge2$,
the theorem of Nomizu--Rodr\'{i}guez \cite{NR} implies that $M$
is a round sphere or an affine $m$-plane. Compactness excludes
the latter. The case $m=1$ follows from the characterization
of taut embedded circles \cite[Theorem~8.3]{Cecil1997}.
\end{proof}

\begin{corollary}[Products of spheres for hypersurfaces]
\label{fin:cor:homology}
Let $M^m\subset\R^{m+1}$ be a compact connected smooth
embedded hypersurface satisfying \eqref{eq:D}.
If $M$ has the $\Z_2$-homology of $S^p\times S^{m-p}$,
where $1\le p<m$, then $M$ is a Dupin cyclide,
diffeomorphic to that product.
\end{corollary}

\begin{proof}
By Proposition~\ref{thm:B}, $M$ is taut. The conclusion is the
homological characterization of taut cyclides due to
Cecil--Ryan \cite{CR}, recalled in
\cite[Theorem~8.6]{Cecil1997}. Their theorem requires only the
stated homology, not an a priori diffeomorphism to the product.
\end{proof}

\begin{remark}\label{fin:rem:productcodimension}
The hypersurface assumption in Corollary~\ref{fin:cor:homology}
is essential. The standard projective embedding
$\mathbb RP^3\to\operatorname{Sym}_0(4,\R)\cong\R^9$,
given by $[x]\mapsto xx^{\mathsf T}-\tfrac14 I$ for $|x|=1$,
is taut for the Frobenius inner product. 
Indeed, $L_B([x])=c_B-x^{\mathsf T}Bx$,
where $c_B$ is a constant depending only on $B$.
For $B$ with distinct eigenvalues this has one critical point
of each index $0,1,2,3$, corresponding to the four eigenlines.
Repeated eigenvalues give degenerate critical points, so every
Morse squared distance function is $\Z_2$-perfect. By
\cite{Pinkall}, this embedding satisfies \eqref{eq:Dnu}.
Yet $\mathbb RP^3$ has the same $\Z_2$-homology groups as
$S^1\times S^2$ and is not diffeomorphic to that product.
Their fundamental groups are $\Z_2$ and $\Z$, respectively.
\end{remark}

In dimension three we use the classification of
Pinkall--Thorbergsson \cite{PT3},
also recalled in \cite[Theorem~8.8]{Cecil1997}.

\begin{corollary}[Dimension three]\label{fin:cor:dimensionthree}
A compact connected smooth embedded submanifold
$M^3\subset\R^n$ satisfying \eqref{eq:Dnu} is
diffeomorphic to one of
\[
 S^3,\quad \mathbb RP^3,\quad S^3/Q_8,\quad T^3,\quad
 S^1\times S^2,\quad S^1\times\mathbb RP^2,\quad
 S^2\widetilde{\times}S^1.
\]
Here $Q_8=\{\pm1,\pm i,\pm j,\pm k\}$ acts on the
unit quaternions by multiplication, and
$S^2\widetilde{\times}S^1$ denotes the nonorientable
$S^2$-bundle over $S^1$. All seven types occur in Euclidean
spaces of suitable dimensions.

If $M$ is a hypersurface in $\R^4$, the list reduces to
\[
 S^3,\qquad S^1\times S^2,\qquad S^3/Q_8,
 \qquad T^3,
\]
and all four types occur.
\end{corollary}

\begin{proof}
Theorem~\ref{thm:general} gives tautness, so the classification
in \cite{PT3} applies. Conversely,
each of the seven diffeomorphism types admits
taut embeddings, which
necessarily satisfy \eqref{eq:Dnu} by
Pinkall's theorem \cite{Pinkall}.
In the hypersurface case,
orientability excludes $S^1\times\mathbb RP^2$ and the
nonorientable $S^2$-bundle over $S^1$.
Corollary~\ref{fin:cor:homology} excludes $\mathbb RP^3$,
since it has the $\Z_2$-homology of $S^1\times S^2$ but is
not diffeomorphic to that product. The four remaining types
are realized by a round sphere, a cyclide, a stereographic
image of Cartan's hypersurface in $S^4$, and the tube
construction of Proposition~\ref{fin:prop:tori}, respectively.
\end{proof}

Further consequences in dimensions two and four follow from
the classical classifications once Theorem~\ref{thm:general}
has established tautness for a compact connected smooth
embedded submanifold $M^m\subset\R^n$ satisfying
\eqref{eq:Dnu}. In dimension two, this implies that
$M$ is diffeomorphic to $S^2$, $T^2$, or
$\mathbb RP^2$ \cite[Theorem~8.1]{Cecil1997}.

In dimension four, the possible diffeomorphism types
can be further restricted by the Betti numbers
$b_i=\dim_{\Z_2}H_i(M;\Z_2)$. If $b_1=0$, then $M$
is diffeomorphic to $S^4$, $S^2\times S^2$, or
$\mathbb CP^2$. If $b_2=0$, then $M$ is
diffeomorphic to $S^4$ or $S^1\times S^3$
\cite[Theorems~1 and~2]{Gor}.

\section{Examples and scope}\label{fin:sec:scope}

For compact connected embedded proper Dupin hypersurfaces,
the classical restrictions on the number $s$ of
distinct principal curvatures give
\[
s\in\{1,2,3,4,6\}
\ \ \text{ with }\ \ \beta(M)=2s\le12,
\]
and the principal multiplicities agree with those of
isoparametric hypersurfaces
\cite{Thorbergsson,GH,Stolz}.

The following standard construction shows why these
restrictions do not extend to all hypersurfaces
satisfying \eqref{eq:D}.

\begin{proposition}[Iterated tubes]\label{fin:prop:tori}
For every $m\ge1$ there is a compact connected embedded taut,
and hence Dupin, hypersurface $M_m\subset\R^{m+1}$ diffeomorphic
to $T^m$. It has $m$ distinct principal curvatures on an open
dense subset. For $m\ge3$ it also has points with just two
distinct principal curvatures, and is therefore not proper
Dupin. Consequently any bound as in Corollary~\ref{thm:C} must
satisfy
\[
 B(m)\ge2^m.
\]
\end{proposition}

\begin{proof}
Start with a round circle $M_1\subset\R^2$.
Given $M_j\subset\R^{j+1}$, include the ambient
space as a hyperplane in $\R^{j+2}$ and let
$M_{j+1}$ be a sufficiently small tube about $M_j$.
Adding the ambient direction preserves tautness,
and the tube is taut by \cite{Pinkall}.

The global unit normal field $\nu$ of $M_j$,
together with a constant unit vector $\zeta$
orthogonal to $\R^{j+1}$, trivializes its
rank-two normal bundle. Thus
$M_{j+1}\cong M_j\times S^1$, giving
$M_j\cong T^j$.

It remains to examine the principal curvatures
of the hypersurfaces obtained by this construction.
Let $A$ denote the shape operator of $M_j$
with respect to $\nu$. Since $\zeta$ is constant,
the shape operator corresponding to
\[
\xi=\cos\theta\,\nu+\sin\theta\,\zeta
\]
is given by $A_\xi=\cos\theta\,A$.

Applying the tube formula
\eqref{eq:general-tube-curvatures}, we obtain
the horizontal principal curvatures and the
additional curvature in the circle direction
\begin{equation}\label{fin:eq:tubecurvature}
\frac{\cos\theta\,\lambda}
     {1-\varepsilon\cos\theta\,\lambda}
\qquad\text{and}\qquad
-\frac1\varepsilon,
\end{equation}
where $\lambda$ runs through the principal
curvatures of $M_j$.
Choosing $\varepsilon$ sufficiently small at each
inductive step ensures that all denominators
remain positive. 
Whenever $\cos\theta\ne0$,
distinct principal curvatures of $M_j$ give
distinct horizontal principal curvatures of
$M_{j+1}$, none of which equals $-1/\varepsilon$.
By the induction hypothesis, $M_j$ has $j$
distinct principal curvatures on an open dense
subset. Together with the condition
$\cos\theta\ne0$, this shows that $M_{j+1}$
has $j+1$ distinct principal curvatures on
an open dense subset.
On the other hand, at $\cos\theta=0$ all
horizontal principal curvatures vanish, so
that only $0$ and $-1/\varepsilon$ remain.

Induction therefore proves the asserted
curvature properties. Finally, since
$M_m\cong T^m$, we have
\[
\beta(M_m)=\beta(T^m)=2^m,
\]
which gives the claimed lower bound.
\end{proof}

The examples have $\beta=8$ with generically three principal
curvatures for $m=3$, and five distinct principal curvatures
on an open dense set for $m=5$. Thus neither $\beta=2s$
nor the proper-Dupin list of curvature counts applies
without properness. Corollary~\ref{thm:C} also applies to
the compact proper Dupin examples of Pinkall--Thorbergsson
\cite{PTdef} and Miyaoka--Ozawa \cite{MO}, which are not
Lie-equivalent to isoparametric hypersurfaces
\cite[Section~6]{Cecil}.

For hypersurfaces already known to be taut, Chi's analyticity
theorem \cite{Chi} can replace the continuation argument.
The polynomial relation of uniformly bounded degree supplied
by Proposition~\ref{prop:degree} on a single proper Dupin
patch extends to the connected hypersurface by the analytic
identity theorem. This uses the local degree bound extracted
from \cite[Section~3]{CCJ}. The carrier and tube arguments
then give uniform semialgebraic description complexity,
to which the classical Nash finiteness results of
Coste--Shiota \cite{CS} apply. Preserving tautness along
the isotopies further uses the semialgebraicity of the
taut parameter sets established in
Section~\ref{fin:sec:deformations}. Related finiteness
results for isoparametric and equifocal hypersurfaces
appear in \cite{Wu,GQ}.

\begingroup
\raggedright

\endgroup
\end{document}